\documentclass[11pt]{article}

\usepackage{geometry}
\usepackage[english]{babel}
\usepackage[utf8x]{inputenc}

\newcommand{\cupall}{\pmb{\pmb{\cup}}}

\usepackage{titling}

\usepackage{amsmath,amsthm,amsfonts,amssymb,epsfig,color,cite}

\usepackage{graphicx}  
\usepackage{bm}

\usepackage{enumerate}

\usepackage[bookmarks,breaklinks]{hyperref}

\newcommand{\remove}[1]{}

\newcommand{\intr}{\operatorname{\bf int}}
\newcommand{\bigmid}{\;\big|\;}
\newcommand{\w}{\operatorname{{\bf w}}}

\newcommand{\PDP}{\textsc{PDPP}}
\newcommand{\DP}{\textsc{DPP}}

\newcommand{\trim}{{\mathbf{trim}}}

\newcommand{\tw}{{\mathbf{tw}}}

\newcommand{\clos}{{\mathbf{clos}}}

\renewcommand{\int}{{\mathbf{int}}}

\newcommand{\bnd}{\textsc{\bf bnd}}

\theoremstyle{remark}

\theoremstyle{plain}

\newtheorem{theorem}{Theorem}
\newtheorem{lemma}{Lemma}
\newtheorem{observation}{Observation}

\newtheorem{proposition}{Proposition}

\begin{document}
\thanksmarkseries{alph}
\title{Irrelevant Vertices for the Planar Disjoint Paths Problem\thanks{{Emails:  Isolde Adler: {\sf I.M.Adler@leeds.ac.uk},
Stavros Kolliopoulos: {\sf sgk@di.uoa.gr},
Philipp Klaus Krause: {\sf philipp@informatik.uni-frankfurt.de},
Daniel Lokshtanov: {\sf daniello@ii.uib.no}, 
Saket Saurabh: {\sf saket@imsc.res.in}, 
Dimitrios  M. Thilikos: {\sf sedthilk@thilikos.info}}\ .%
}
}
\author{Isolde Adler\thanks{%
School of Computing.
University of Leeds, 
UK.
}
\and Stavros G. Kolliopoulos\thanks{Department of Informatics and Telecommunications, National and Kapodistrian University of Athens, Athens, Greece.}\ \thanks{Co-financed by the European Union (European Social Fund -- ESF) and
Greek national funds through the Operational Program ``Education and Lifelong Learning'' of the
National Strategic Reference Framework (NSRF) - Research Funding Program:
``{\sl Thalis. Investing in knowledge society through the European Social Fund}''.}
\and Philipp Klaus Krause$^{\mbox{\footnotesize b}}$\thanks{Supported by a fellowship within the FIT-Programme of the German Academic Exchange Service (DAAD) at NII, Tokyo and by DFG-Projekt GalA, grant number AD 411/1-1.}
\and Daniel Lokshtanov\thanks{Department of Informatics, University of Bergen, Norway.} \thanks{Supported by ``Rigorous Theory of Preprocessing, ERC Advanced Investigator Grant 267959" and ``Parameterized Approximation, ERC Starting Grant 306992".} 
\and Saket Saurabh\thanks{The Institute of Mathematical Sciences, CIT Campus, Chennai, India.}\,  $^{\mbox{\footnotesize f\! g}}$
\and Dimitrios  M. Thilikos\thanks{Department of Mathematics, National and Kapodistrian University of Athens, Athens, Greece.}\ \!\! \thanks{AlGCo project-team, CNRS, LIRMM, France.}\ $^{\mbox{\footnotesize d}}$
}

\date{\empty}
 
\maketitle
\begin{abstract}
\noindent 
The {\sc Disjoint Paths Problem} asks, given a graph $G$ and a 
set of pairs of terminals $(s_{1},t_{1}),\ldots,(s_{k},t_{k})$, whether 
there is a collection of $k$ pairwise vertex-disjoint paths linking  
$s_{i}$ and $t_{i}$, for $i=1,\ldots,k.$
In their $f(k)\cdot n^{3}$  algorithm for this problem, Robertson and Seymour
introduced the {\sl irrelevant vertex technique} according to 
which in every instance of treewidth greater than $g(k)$ there is an ``irrelevant'' vertex
whose removal creates an equivalent instance of the problem. 
This fact is based on the celebrated {\sl Unique Linkage Theorem}, 
whose -- very technical --  proof gives a function $g(k)$ that is 
responsible for an  immense parameter dependence in the 
running time of the algorithm.
In this paper we give a new and self-contained proof of this result that strongly exploits 
the combinatorial properties of planar graphs and achieves  $g(k)=O(k^{3/2}\cdot 2^{k}).$ 
Our bound is radically 
better than the bounds known for general graphs.
\end{abstract}


\section{Introduction}

One of the most studied problems in graph theory is the
 {\sc Disjoint Paths Problem (\DP)}: {\em Given a graph $G$ and a set ${\cal P}$ of $k$ pairs 
of terminals, $(s_{1},t_{1}),\ldots,$
$(s_{k},t_{k})$, decide whether $G$ 
contains $k$ vertex-disjoint paths  $P_{1},\ldots,P_{k}$ where $P_{i}$
has endpoints $s_{i}$ and $t_{i}$, $i=1,\ldots,k$}.  
In addition to  its numerous applications in areas such as network
routing and VLSI layout, this problem has  
been the catalyst for extensive research in algorithms and
combinatorics \cite{Schrijver03}.
{\DP}
is  {\sf NP}-complete, 
along with its  edge-disjoint or  
directed variants, even when the input graph is planar~\cite{Vygen95npco,MiddendorfP93onth,KramerL84thec,Lynch75}.
The celebrated algorithm of Roberson and Seymour 
solves it  however in $f(k)\cdot n^{3}$ steps, where $f$  is some computable 
function~\cite{RobertsonS-GMXIII}. This implies that, when we parameterize 
{\DP} by the number $k$ of pairs of terminals, the
problem is 
fixed-parameter tractable.
The Robertson-Seymour  algorithm  is the central algorithmic result of the Graph 
Minors series of papers,  one of the deepest and most influential bodies of work in graph
theory. 

The basis of the algorithm in~\cite{RobertsonS-GMXIII} is the 
so-called  {\em irrelevant-vertex technique}  which can be summarized very
roughly as follows.
As long as  the input graph $G$ violates certain structural conditions, it is possible
to find a vertex $v$ that is {\em solution-irrelevant:} 
every collection of paths certifying 
a solution to the problem can be  rerouted to an {\em equivalent} one,
that links the same pairs of terminals,  but in which the new paths avoid  $v.$ 
One then iteratively removes
such irrelevant vertices until the structural conditions 
are met. By that point the  graph has been simplified enough so that the problem
can be attacked via dynamic programming. 
  
The following two structural conditions are used by the algorithm
in~\cite{RobertsonS-GMXIII}: 
{\sf (i)} $G$ excludes a clique, whose size depends on $k$,
as a minor and {\sf (ii)} $G$ has  treewidth bounded by some function of $k.$ 
When it comes to  enforcing Condition  ({\sf ii}), the 
aim  is to prove that in graphs without 
big clique-minors and with treewidth at least $g(k)$ there is always 
a solution-irrelevant vertex. This is the most 
complicated part of the proof and it was
postponed until the later 
papers in the series \cite{RobertsonS-XXI,RobertsonSGM22}. 
The bad news is that the complicated proofs also imply 
an {\sl immense}  parametric
dependence, as
expressed by the function $f,$  of the running time on the parameter $k.$ 
This puts the algorithm  outside the realm of feasibility even for 
elementary values of $k.$ 
%

The ideas above were powerful enough  to be applicable  also to  problems 
outside the context of the Graph Minors series. 
During the last decade, they have   been applied to  many other combinatorial problems and now
they  constitute a basic paradigm
in parameterized algorithm design (see, e.g., 
\cite{DawarGK07,DawarK09domi,GolovachKPT09,KawarabayashiK08,Kawarabayashi:2010cs,Kobayashi:2009jt}). 
However, in most applications, the need for overcoming 
the high parameter dependence emerging from the 
structural theorems of the Graph Minors series, especially  those 
in \cite{RobertsonS-XXI,RobertsonSGM22}, remains imperative. 
Hence two natural directions of research are: simplify parts 
of the original proof for the general case or focus on specific graph classes 
that may admit proofs with  better parameter dependence.  An 
important  step in the first direction
was taken recently by Kawarabayashi and Wollan in~\cite{KawarabayashiW2010asho} who gave an easier 
and shorter proof of the results in \cite{RobertsonS-XXI,RobertsonSGM22}.
While the parameter  dependence of the new proof is  certainly much
better than the previous, immense, function, it is still huge: a rough estimation from~\cite{KawarabayashiW2010asho}
gives a lower bound for $g(k)$ of magnitude  $2^{2^{{{2^{\Omega(k)}}}}}$ which in turn implies 
a lower bound for $f(k)$ of magnitude ${2^{2^{2^{{2^{\Omega(k)}}}}}}.$

In this paper we offer a solid advance in the second direction, focusing on planar graphs (see also~\cite{ReedRSS91find,Schrijver94} for previous results on planar graphs).
We show  that, for planar graphs, $g(k)$ is single exponential. In particular we prove the following result.
\begin{theorem}
\label{main}
Every instance of \DP\ consisting 
of a planar graph $G$ with treewidth at least $82\cdot k^{3/2}\cdot 2^{k}$ and $k$ pairs of terminals
contains a vertex $v$ such that every solution to {\DP} 
can be replaced by an equivalent one whose paths avoid~$v.$ 
\end{theorem}
%

The  proof  of Theorem~\ref{main}  is presented in Section~\ref{sec:ub} and deviates  significantly from  those
in~\cite{RobertsonS-XXI,RobertsonSGM22,KawarabayashiW2010asho}.  It  is
self-contained  and exploits extensively  the combinatorics  of planar
graphs.  Given a  \DP\ instance defined on a planar graph $G,$ we prove that if $G$ 
contains as a subgraph a 
subdivision of a sufficiently large (exponential in $k$)  grid, 
whose ``perimeter'' 
does not enclose any terminal, then the ``central'' vertex $v$ of the grid
is  solution-irrelevant for this instance. It follows that 
the ``area'' provided by the grid  is big enough so that every solution
 that uses $v$   
can be rerouted to an equivalent one that does not go so deep in 
the grid and therefore avoids the vertex $v$.

Combining Theorem~\ref{main} with
 known algorithmic results, it is possible 
to reduce, in $2^{2^{O(k)}}\cdot n^2$ steps, a planar 
instance of \DP\  to an equivalent one whose graph has treewidth  $2^{O(k)}.$
Then, using standard dynamic programming on tree decompositions, 
a solution, if one exists, can be found 
in $2^{2^{O(k)}}\cdot n $ steps.
The  parametric dependence of this algorithm is a step forward
in the study of the parameterized complexity of \DP\ on planar graphs.
This algorithm is abstracted in the following theorem, whose proof is in Section~\ref{sec:algo}.
%
%
%
%

\begin{theorem}   \label{theorem:algo}
There exists an algorithm 
that, given an instance $(G,{\cal P})$ of \DP, 
where $G$ is a planar  $n$-vertex graph and $|{\cal P}|=k$,  
either reports that $(G,{\cal P})$ is a NO-instance or outputs 
a solution of \DP\ for $(G,{\cal P})$. This algorithm runs in
$2^{2^{O(k)}}\cdot n^2$ steps.
\end{theorem}

An extended abstract of this  work, without any proofs, appeared 
in \cite{AdlerKKLST11tigh}.  
Some of our ideas  have  proved useful  in  the recent
breakthrough  result of  Cygan et  al.  that establishes fixed-parameter
tractability for $k$-disjoint paths on planar directed graphs~\cite{CyganMPP13thep}.

%
%



\section{Basic definitions}
\label{sec:def}

Throughout this paper, given a collection of sets ${\cal C}$ we denote by 
$\cupall {\cal C}$  the set $\cup_{x\in {\cal C}}x$, i.e., the union of all sets in ${\cal C}$.

All graphs that we consider  are finite, undirected, and simple. We denote the vertex set of a graph
$G$ by $V(G)$ and the edge set by $E(G).$ Every edge is a two-element subset of
$V(G).$ A graph $H$ is a \emph{subgraph} of a graph $G$, denoted by $H\subseteq G$, if $V(H)\subseteq V(G)$ and
$E(H)\subseteq E(G).$ Given two graphs $G$ and $H$, we define $G\cap H=(V(G)\cap V(H), E(G)\cap V(H))$ and  $G\cup H=(V(G)\cup V(H), E(G)\cup V(H))$. Given a $S\subseteq V(G)$, we also denote by $G[S]$ the subgraph 
of $G$ induced by $S$.


A \emph{path} in a graph $G$ is a  
connected acyclic subgraph  with at least one vertex
whose vertices have degree at most 2. The {\em length}
of a path $P$ is equal to the number of its edges. The {\em endpoints}
of a path $P$ are its vertices of degree 1 (in the trivial 
case where there is only one endpoint $x$, we 
 say that the endpoints of $P$ are $x$ and $x$).
An $(x,y)$-path of $G$ is any path of $G$ whose endpoints are $x$ and $y$.

A {\em cycle} of a graph $G$ is a connected subgraph of $G$
whose vertices have  degree 2. 
%
For graphs $G$ and $H$ the \emph{cartesian product} is the graph whose vertex set is 
$V(G) \times V(H)$ and  whose edge set is $\{\{(v, v'),(w, w')\} \mid  (\{v,w\} \in E(G) \text{ $\wedge$ } v' = w') \text{ $\vee$ } (v = w \text{ $\wedge$ } \{v',w'\} \in E(H))\}.$
%

\paragraph{The {\sc Disjoint Paths} problem.}
The problem that we examine in this paper is the following.\\

\fbox{\begin{minipage}{13cm}
{\sc Disjoint Paths (DPP)} \\
	{\sl Input}: A graph $G$, and  a collection  ${\cal P}=\{(s_{i},t_{i})\in V(G)^{2}, i\in\{1,\ldots,k\}\}$
of pairs of $2k$  terminals of $G$. \\
	{\sl Question}: Are there $k$ pairwise vertex-disjoint paths $P_1,\ldots,P_k$ in $G$
		such that for $i\in\{1,\ldots,k\}$, $P_i$ has endpoints $s_i$ and $t_i$?
\end{minipage}} \medskip\medskip\medskip \\
%
We call the $k$-pairwise  vertex-disjoint paths
certifying a YES-instance of \DP\ a {\em solution} of \DP\
for the input $(G,{\cal P})$.
Given an instance $(G,{\cal P})$ of \DP,
we say that a non-terminal vertex $v\in V(G)$ is {\em irrelevant for} $(G,{\cal P})$,  
if $(G,{\cal P})$ is a YES-instance if and only if 
$(G\setminus v,{\cal P})$ is a YES-instance.
We denote by  \PDP\ the restriction of \DP\ on 
instances $(G,{\cal P})$ where $G$ is a planar graph.

\paragraph{Minors.}
 A graph $H$ is a \emph{minor} of a graph $G$, if 
 there is a function $\phi: V(H)\rightarrow 2^{V(G)}$,  such that
 \begin{itemize}
 \item[$i.$] For every two distinct vertices $x$ and $y$ of $H$, $G[\phi(x)]$ and $G[\phi(y)]$
 are two vertex-disjoint connected subgraphs of $G$ and
 \item[$ii.$] for every two adjacent vertex $x$ and $y$ of $H$, $G[\phi(x)\cup \phi(y)]$ is a
 connected subgraph of $G$. 
 \end{itemize}
 We call the function $\phi$ {\em minor model} of $H$ in $G$.

%




\paragraph{Grids.}
Let $m,n\geq 1.$ The $(m\times
n$)-\emph{grid} is the Cartesian product of a path of length $m-1$ and
a path of length $n-1.$ In the case of a {\em square grid} where $m=n$,  we say that $n$ is the {\em size} of the grid. Given that $n,m\geq 2$, the {\em corners} of an $(m\times
n$)-\emph{grid} are its vertices of degree 2.  When we refer to a $(m\times
n$)-\emph{grid} we will always assume an orthogonal orientation of 
it that classifies its corners to the {\em upper left}, {\em upper right}, 
{\em down right}, and 
{\em down left} corner of it.
 
Given that $\Gamma$ is an  $(m\times
n$)-\emph{grid},  we say that a vertex of $G$ is one of its {\em centers}
if its distance from the set of its corners is the maximum possible. Observe 
that a square grid of even size has exactly $4$ centers.
We also consider an $(m\times
n$)-{grid} embedded in the plane so that, if it has more than $2$ faces then the
infinite one is incident to more than 4 vertices. The {\em outer cycle} of an embedding of an $(m\times
n$)-{grid}  is the one that is the boundary of its infinite face.
We also refer to the {\em horizontal} and the {\em vertical lines}
of an $(m\times
n$)-\emph{grid} as its paths between vertices of degree smaller than 4
that are traversing it either ``horizontally'' or ``vertically'' respectively.  We make the convention that an $(m\times
n$)-{grid} contains $m$ vertical lines and $n$ horizontal lines.
The {\em lower horizontal line}  and the {\em higher horizontal line} 
of $\Gamma$ are defined in the obvious way
 (see Figure~\ref{fig:tgoop}
for an example).

\begin{figure}[ht]
\begin{center}
\includegraphics[width=12.3cm]{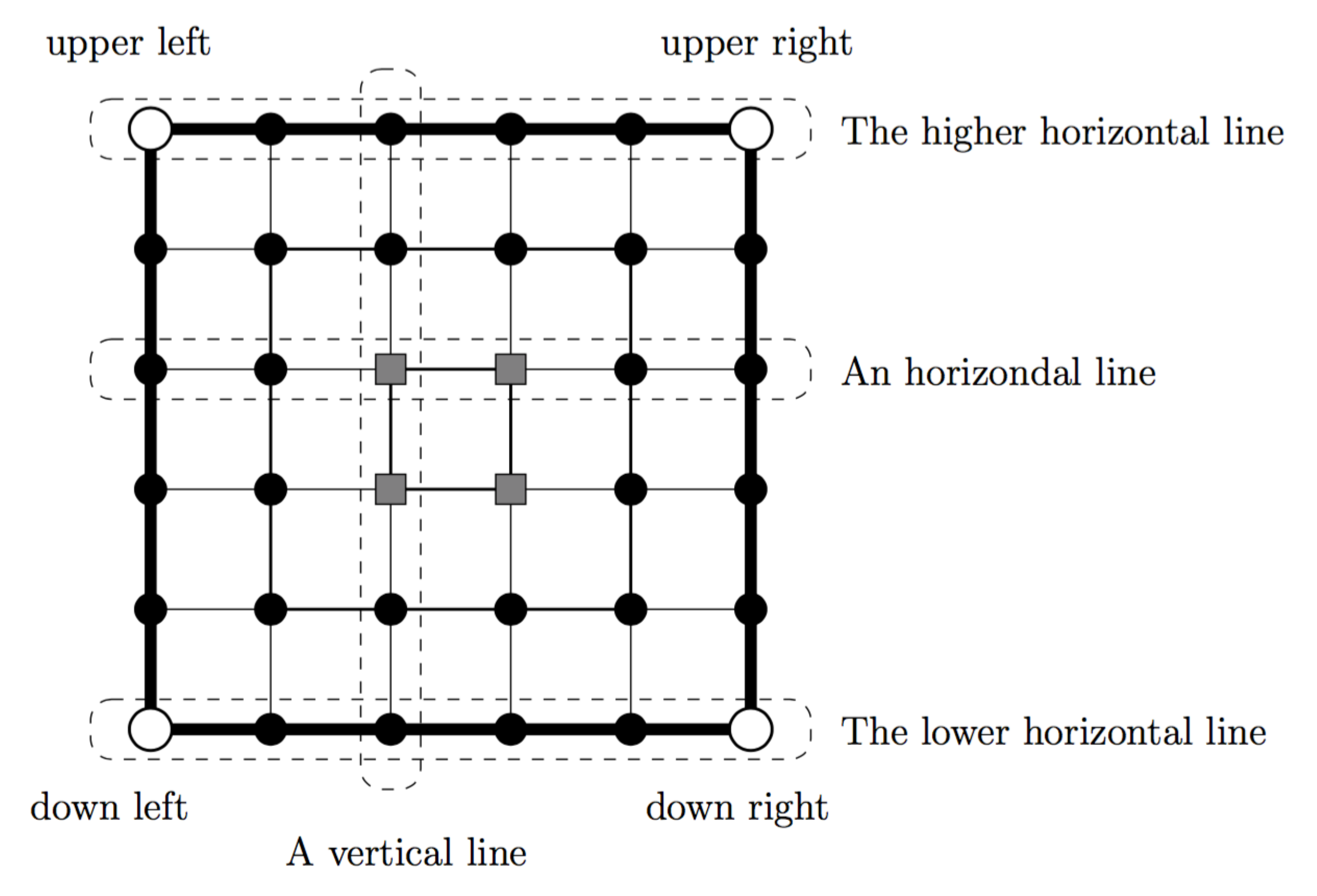}
\end{center}
\caption{A drawing of the $(6\times 6)$-grid. 
The four white round vertices are its corners and the four grey square vertices are its centers.
The cycle formed by the ``fat'' edges is the outer cycle.}
\label{fig:tgoop}
\end{figure}

\paragraph{Plane graphs}
Whenever we refer to a planar graph $G$ we consider an 
embedding of $G$ in the plane $\Sigma=\Bbb{R}^{2}$. 
To simplify notation, we do not distinguish
between a vertex of $G$ and the point of $\Sigma$ used in the
drawing to represent the vertex or between an edge and the arc
representing it.  
We also consider a plane graph $G$ 
as the union of the points corresponding to its vertices
and edges. That way, edges and faces are considered to be open sets of $\Sigma$.
Moreover,  a subgraph $H$ of $G$ can be seen as a graph
$H$, where the points corresponding to $H$ are a subset of the points corresponding to $G$.

Recall that $\Delta \subseteq \Sigma$ is
an open  (resp. closed)  disc if it is homeomorphic to
$\{(x,y):x^2 +y^2< 1\}$ (resp. $\{(x,y):x^2 +y^2\leq 1\}$).
Given a cycle $C$ of $G$ we define its {\em open-interior} (resp. {\em open-exterior})
as the connected component of $\Sigma\setminus C$ that is disjoint 
from (resp. contains)
the infinite face of $G$. The {\em closed-interior} (resp. {\em closed-exterior}) 
of $C$
is the closure of its {\em open-interior} (resp. {\em open-exterior}).
Given a set $A\subseteq \Sigma,$ we denote its {\em interior} (resp. {\em closure}) by $\int(A)$ (resp. $\clos(A)$).
An {\em open (resp. closed) arc} $I$ in $\Bbb{R}^2$ is any set homeomorphic to the set $\{(x,0)\mid x\in(0,1)\}$
(resp.  $\{(x,0)\mid x\in[0,1]\}$)
and the {\em endpoints} of $I$ are defined in the obvious way.
We also define $\trim(I)$ as the set of all points of the arc $I$ except for its endpoints.
%


\paragraph{Outerplanar graphs.}
An  \emph{outerplanar} graph is a plane graph 
whose vertices are all incident to the infinite face. 
If an edge of an outerplanar graph is incident 
to its infinite face then we call it {\em external}, otherwise we
call it {\em internal}.
The {\em weak dual} of an
outerplanar graph $G$ is the graph obtained from the dual of $G$
after removing the vertex corresponding to the infinite face of the embedding.
Notice that if the outerplanar graph $G$ is biconnected, then its weak dual is a tree.
We call a face of an outerplanar graph {\em simplicial} if 
it corresponds to  a leaf of the graph's  weak dual.


\paragraph{Treewidth.}
A \emph{tree decomposition} of a graph $G$ is a pair 
$(T,\chi)$, consisting of a rooted
tree 
$T$ and a mapping $\chi\colon V(T)\to 2^{V(G)}$, such that
for each $v \in V(G)$ there exists $t \in V(T)$ with $v \in \chi(t)$,
for each edge $e \in E(G)$ there exists  a node $t \in V(T)$ with 
$e \subseteq \chi(t)$, and for each $v \in V(G)$ the set
$\{t \in V(T) \mid v \in \chi(t) \}$ is connected in $T.$
%
%
%

The \emph{width} of $(T,\chi)$ is defined as
$\w(T,\chi):= 
\max\big\{\left|\chi(t)\right|-1\ \bigmid t\in V(T)\big\}.$

The \emph{tree-width of $G$} is defined as
\[
	\tw(G):= 
	\min\big\{\w(T,\chi)\ \bigmid (T,\chi) \text{ is a tree decomposition of }G\big\}.
\]

%
%
%

We need the next proposition that follows directly by combining the 
main result of~\cite{GuT12impr}  and 
 (5.1) from~\cite{RobertsonS91GMX}.

\begin{proposition}
\label{gutam}
If  $G$ is a planar graph and ${\bf tw}(G)\geq 4.5\cdot k+1$, then 
$G$ contains a $(k\times k)$-grid as a mimor.
\end{proposition}

Our algorithmic results require  the following proposition. It follows 
from the main result of~\cite{PerkovicR99anim} (see also Algorithm (3.3) in~\cite{RobertsonS-GMXIII}). The parametric  dependence of $k$ in the running time 
follows because the algorithm in~\cite{PerkovicR99anim} uses as a subroutine 
the algorithm in~\cite{Bod96alin} that runs in $2^{k^{O(1)}}\cdot n$  steps.

\begin{proposition}
\label{besdtop}
There exists an algorithm that, given an $n$-vertex graph $G$ and a positive integer $k$,
either outputs a tree decomposition of $G$ of width at most $k$ or 
outputs a subgraph $G'$ of $G$ with treewidth greater than $k$ and 
a tree decomposition of $G'$ of width at most  $2k,$ in $2^{k^{O(1)}}\cdot n$  steps.
\end{proposition}

%

%
%

%

%

\section{Irrelevant vertices in graphs of large treewidth}  \label{sec:ub}

In this section we prove our main result, namely   Theorem~\ref{main}.
We introduce the notion of cheap linkages and explore  their structural properties  
in  Subsections~\ref{subsec:cheap} and \ref{subsec:bounds}. 
In Subsection~\ref{subsec:irre} we bring together the structural
results to show the existence of an irrelevant vertex in a graph of
large treewidth.

\subsection{Configurations and cheap linkages}
\label{subsec:cheap}

In this subsection we introduce some basic definitions on planar graphs
that are necessary for our proof.

\paragraph{Tight concentric cycles.}

Let $G$ be a plane graph and let $D$ be a disk that is the closed interior of some cycle $C$ of $G$. We say that $D$ is {\em internally chordless} if there is no 
path in $G$ whose endpoints are vertices of $C$ and whose edges 
belong to the open interior of $C$.

Let ${\cal C}=\{C_0,\ldots,C_r\},$ be a sequence of cycles in $G.$ 
We denote by   $D_{i}$ the closed-interior of $C_{i}, i\in\{0,\ldots,r\},$
and we say that ${\cal D}=\{D_{0},\ldots,D_{r}\}$ is the {\em disc sequence} of ${\cal C}.$
 We call
${\cal C}$ \emph{concentric}, if for all $i\in\{0,\ldots,r-1\}$, the cycle $C_i$ is contained
in the open-interior of $D_{i+1}.$ The sequence  ${\cal C}$ of concentric cycles  is \emph{tight} in $G$, if, 
in addition,
\begin{itemize}
\item  $D_0$ is {\em internally chordless}. 
\item For every $i\in\{0,\ldots,r-1\}$, there is no cycle 
of $G$ that is contained in  $D_{i+1}\setminus D_i$ and 
 whose closed-interior $D$ has the property  $D_{i}\subsetneq D\subsetneq D_{i+1}.$
\end{itemize}

\begin{figure}[ht]
\begin{center}
\includegraphics[width=6.5cm]{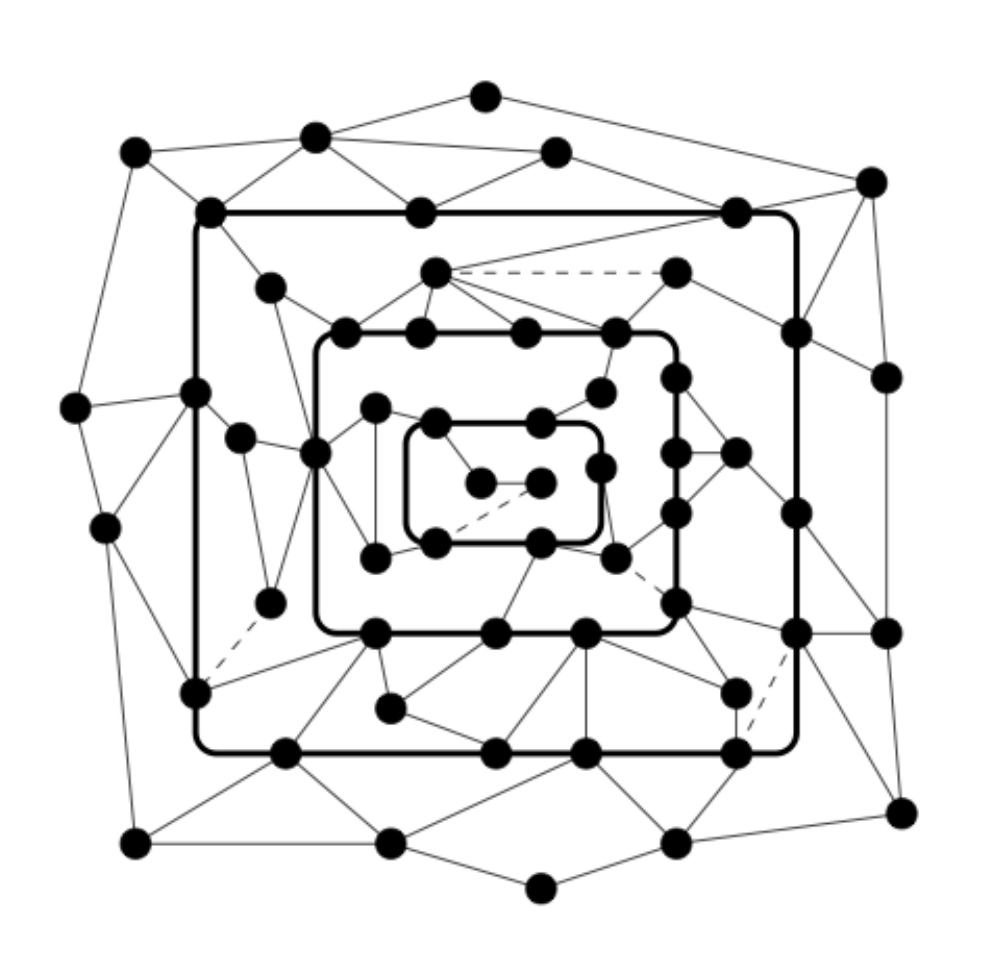}
\end{center}
\caption{An example of a plane graph $G$ and a tight sequence of 3 concentric cycles in it.
Notice that the addition to  $G$ of any of the dashed edges makes this collection of cycles non-tight.}
\label{fig:tight}
\end{figure}

\begin{lemma}
\label{u82mnb1o}
There exists an algorithm 
that given a positive integer $r$,  an $n$-vertex  plane graph $G$, and a $T\subseteq V(G)$, either  outputs a tree decomposition of $G$ of width at most 
$9\cdot (r+1)\cdot \lceil \sqrt{|T|+1}\, \rceil )$ or an internally chordless cycle $C$ 
of $G$ such that  there exists a tight sequence of cycles  $C_{0},\ldots, C_{r}$ in $G$ where 
\begin{itemize}
\item
$C_{0}=C$ and 
\item
all vertices of $T$
are in the open exterior of $C_{r}$.
\end{itemize}
Moreover, this algorithm runs in $2^{(r\cdot \sqrt{|T|})^{O(1)}}\cdot n$ steps.

%
\end{lemma}

\begin{proof}
%
%
%
%
%
%
%
Let $x=|T|+1$ and $y=2(r+1)\cdot \lceil \sqrt{x}\, \rceil $.
From Proposition~\ref{gutam}, if $\tw(G)\geq  4.5\cdot y+1$, then $G$ contains as a minor 
a $(y\times y)$-grid $\Gamma$. 
We now observe that the grid $\Gamma$ contains as subgraphs $x$ 
pairwise disjoint $(2(r+1)\times 2(r+1))$-grids $\Gamma_{1},\ldots,\Gamma_{x}$. Note that
each $\Gamma_{i}, i\in\{1,\ldots,x\}$ contains a sequence of $r+1$ concentric cycles
that, given a minor model $\phi$ of $\Gamma$ in $G$, can be 
used to construct, in linear time, a sequence of  $r+1$ 
concentric cycles ${\cal C}_{i}=\{C_{0}^{i},C_{1}^{i},\ldots,C_{r}^{i}\}$
 in $G$ such that for every $i,j\in\{1,\ldots,x\}$, where  $i\neq j$, 
all cycles in ${\cal C}_{j}$ are in the open exterior of $C_{r}^{i}.$

Note that at least one, say $C_{r}^{i}$, of the cycles in $\{C_{r}^{1},\ldots,C_{r}^{x}\}$
should  contain all the vertices of $T$ in its open exterior. 
Let $e$ be any edge of $C_{0}^{i}$.
Let also $f$ be the face of $G$ that is contained in the 
open interior of $C_{0}^{i}$ and is incident to $e$. 
Let $J_{f}$ be the graph consisting of the 
vertices and the edges that are incident to $f$. It is easy to verify that, 
$J_{f}$
contains an internally 
chordless cycle $C$ that contains the edge $e$. Given $C_{0}^{i}$, 
the cycle $C$ can be found in linear time. Notice now that $G$ contains  
a tight sequence of cycles $C_{0},C_{1},\ldots,C_{r}$ such that $C_{0}=C$
and where, for $h\in\{0,\ldots,r\}$, $C_{h}$ is in the closed interior of $C_{h}^{i}$.
The result follows as the open exterior of $C_{r}$ contains the open exterior of $C_{r}^{i}$
and therefore contains all vertices in $T$.

The algorithm runs as follows: it first uses
the algorithm of  Proposition~\ref{besdtop} for $k=4.5\cdot y$.
If the algorithm outputs a tree decomposition of $G$ of width at most $k$, 
then we are done. Otherwise it outputs a subgraph $G'$ of $G$
where $\tw(G')>k$ and a tree decomposition of $G'$ of width $\leq 2k$.
We use this tree decomposition in order to find a minor model $\phi$ 
of the $(y\times y)$-grid $\Gamma$ in $G'$. This can be done in $2^{k^{O(1)}}=2^{(r\cdot \sqrt{|T|})^{O(1)}}\cdot n$
steps using the algorithm in~\cite{AdlerDFST2010fast} (or, alternatively, the algorithm in~\cite{Hicks04bran}). Clearly, $\phi$ is also a minor model of $\Gamma$ in $G$.
We may now use $\phi$, as explained above, in order to identify, in linear time,
 the required   internally 
chordless cycle $C$ in $G$.
%
%
\end{proof}
%

\paragraph{Linkages.}
A \emph{linkage} in a graph $G$ is a non-empty subgraph $L$ of $G$  whose connected components are all paths. The {\em paths} of a { linkage} are its connected components and we denote them by ${\cal P}(L).$
The \emph{terminals} of a linkage $L$ are the endpoints of the paths in ${\cal P}(L)$, and
the \emph{pattern} of $L$ is the set $\big\{\{s,t\}\mid {\cal P}(L)\text{ contains a path}$ $\text{from $s$ to $t$ in $G$}\big\}.$ Two linkages are {\em equivalent} if they have the same pattern.
%

\paragraph{Segments.}
Let $G$ be a plane graph 
and let $C$ be a cycle in $G$ whose closed-interior is $D$.
Given a path $P$ in $G$ we say that a subpath $P_0$ of $P$ is
a $D$\textup{-segment} of $P$, if $P_0$ is a non-empty (possibly edgeless) path 
obtained by intersecting $P$ with $D.$
For a linkage $ L$ of $G$ we say that
a path $P_0$ is a $D$\emph{-segment} of $L$, if $P_0$ is
a $D$\emph{-segment} of some path $P$ in $\mathcal P(L).$

\begin{figure}[ht]
\begin{center}
\includegraphics[width=9cm]{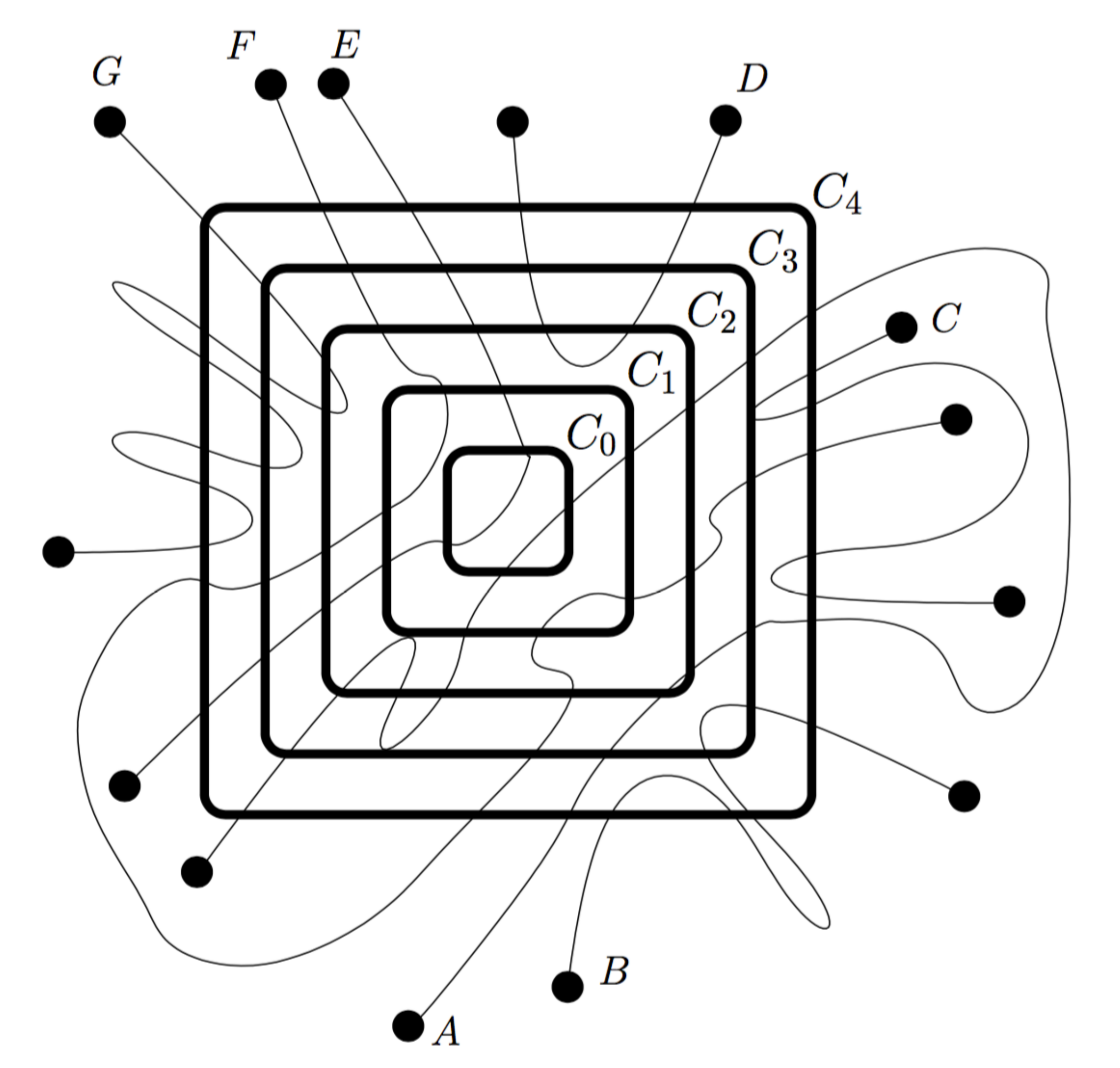}
\end{center}
\caption{An example of a  CL-configuration ${\cal Q}=({\cal C},L)$ where ${\cal C}$ contains $5$ cycles and $L$ has $7$ paths. ${\cal Q}$ has 13 segments. Linkage paths $A$, $B$, $C$, $D$, $E$, $F$, and $G$, contain 2, 2, 2, 1, 1, 2, 3 of these segments respectively. Also the eccentricities of the segments of $A$, are $0$ and $2$, of $B$ are $3$ and $4$.
Notice that one of the two segments of $A$ has two 3-chords, each having 2 semi 3-chords.}
\label{fig:cdL}
\end{figure}

 \paragraph{\mbox{\rm CL}-configurations.}
Given a plane graph $G$, we say that a pair ${\cal Q}=({\cal C},L)$ is a {\em CL-configuration}
of $G$  of {\em depth} $r$ if ${\cal C}=\{C_0,\ldots,C_r\}$ is a 
sequence of concentric cycles in $G,$  $L$ is a linkage of $G,$ and
 $D_r$ does not contain any terminals of $L.$ 
A  \emph{segment} of  ${\cal Q}$ is any $D_r$-segment of $L$. 
The {\em eccentricity} of a segment $P$ of ${\cal Q}$ is the minimum
$i$ such that $V(C_{i}\cap P)\neq\varnothing$.
A segment of  ${\cal Q}$ is {\em extremal} 
if it is has  eccentricity $r$. Observe that 
if ${\cal C}$ is tight then any extremal segment  is a subpath of $C_{r}.$
Given a cycle $C_{i}\in{\cal C}$ and a segment $P$ of ${\cal Q}$
we define the {\em $i$-chords} of $P$ as the connected components of 
$P\cap \int(D_{i})$ (notice that $i$-chords are open arcs).
For every  $i$-chord $X$ of $P$, we define the {\em $i$-semichords}
of $P$ as the connected components of the set $X\setminus D_{i-1}$ (notice 
that $i$-semichords are open arcs).
Given a segment $P$ that does not have any $0$-chord,
we define its {\em zone}  as the connected  component 
of $D_{r}\setminus P$ that does not contain the open-interior of $D_{0}$ (a zone is an open set).

\begin{figure}[ht]
\begin{center}
\includegraphics[width=6cm]{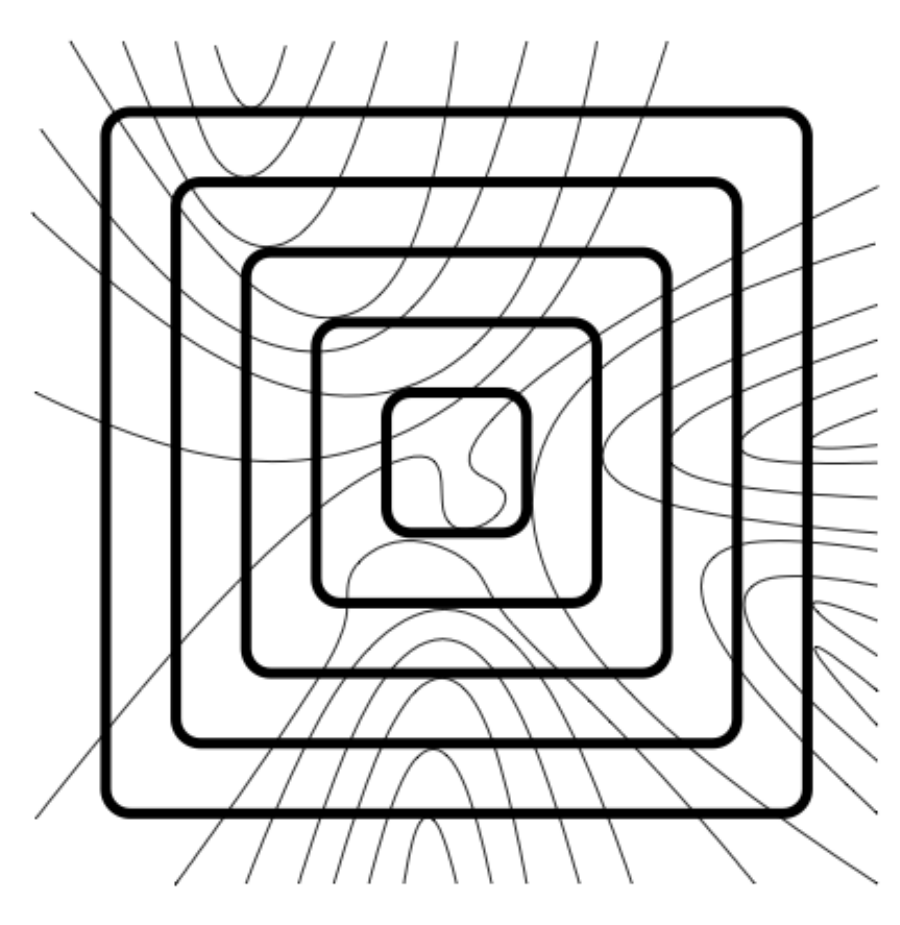}
\end{center}
\caption{An example of a  CL-configuration $({\cal C},L)$ 
where the linkage $L$ is ${\cal C}$-cheap. Only   the 5 concentric cycles of ${\cal C}$ and a cropped part of the linkage $L$ are depicted. Notice that the collection of concentric cycles ${\cal C}$ is not tight.}
\label{fig:cheapCL}
\end{figure}

A CL-configuration ${\cal Q}=({\cal C},L)$ is called {\em reduced} if 
the graph $L\cap\cupall {\cal C}$ is edgeless. 
Let ${\cal Q}=({\cal C},L)$ be a CL-configuration  of $G$
and let $E^{\bullet}$ be the set of all edges of the graph $L\cap\cupall {\cal C}$.
We then define $G^{*}$ as the graph obtained if we contract in $G$ all edges in $E^{\bullet}$. We also define ${\cal Q}^{*}$
as the pair $({\cal C}^{*},L^{*})$  obtained if in $L$ and in the cycles of ${\cal C}$
we contract all edges of $E^{\bullet}$. Notice that ${\cal Q}^{*}$ is a reduced CL-configuration of $G^{*}$. We call $({\cal Q}^*,G^*)$ the {\em reduced pair}
of $G$ and ${\cal Q}$.

\paragraph{Cheap linkages.}
Let $G$ be a plane graph and ${\cal Q}=({\cal C},L)$ be a {CL-configuration}
of $G$ of depth $r$.
We define the function $c: \{L \mid L \text{\ is a linkage of\ }  G\} \to \mathbb{N}$ so  that 
\begin{displaymath}
	c(L)=|E(L)\setminus \bigcup_{i\in\{0,\ldots,r\}}E(C_{i})|.
\end{displaymath}
A linkage $L$ of $G$ is  \textup{${\cal C}$-cheap}, if there is no other  {CL-configuration}
${\cal Q}'=({\cal C},L')$ such that $L'$ has the same pattern as $L$ and $c(L) > c(L').$ 
Intuitively, the function $c$ defined above penalizes every edge of the linkage that does not lie 
on some cycle $C_{i}.$

\begin{observation}
\label{obs1}
Let ${\cal Q}=({\cal C},L)$ be a {CL-configuration} and let $(G^*,{\cal Q}^*=({\cal C}^*,L^*))$ be the reduced pair of $G$ and ${\cal Q}$. Then 
\begin{itemize}
\item If $L$ is ${\cal C}$-cheap, then  $L^*$ is ${\cal C}^*$-cheap.
\item If ${\cal C}$ is tight in $G,$ then ${\cal C}^{*}$ is tight in $G^*$.
\end{itemize}
\end{observation}

%
%

%
%

\begin{figure}[ht]
\begin{center}
\includegraphics[width=6.8cm]{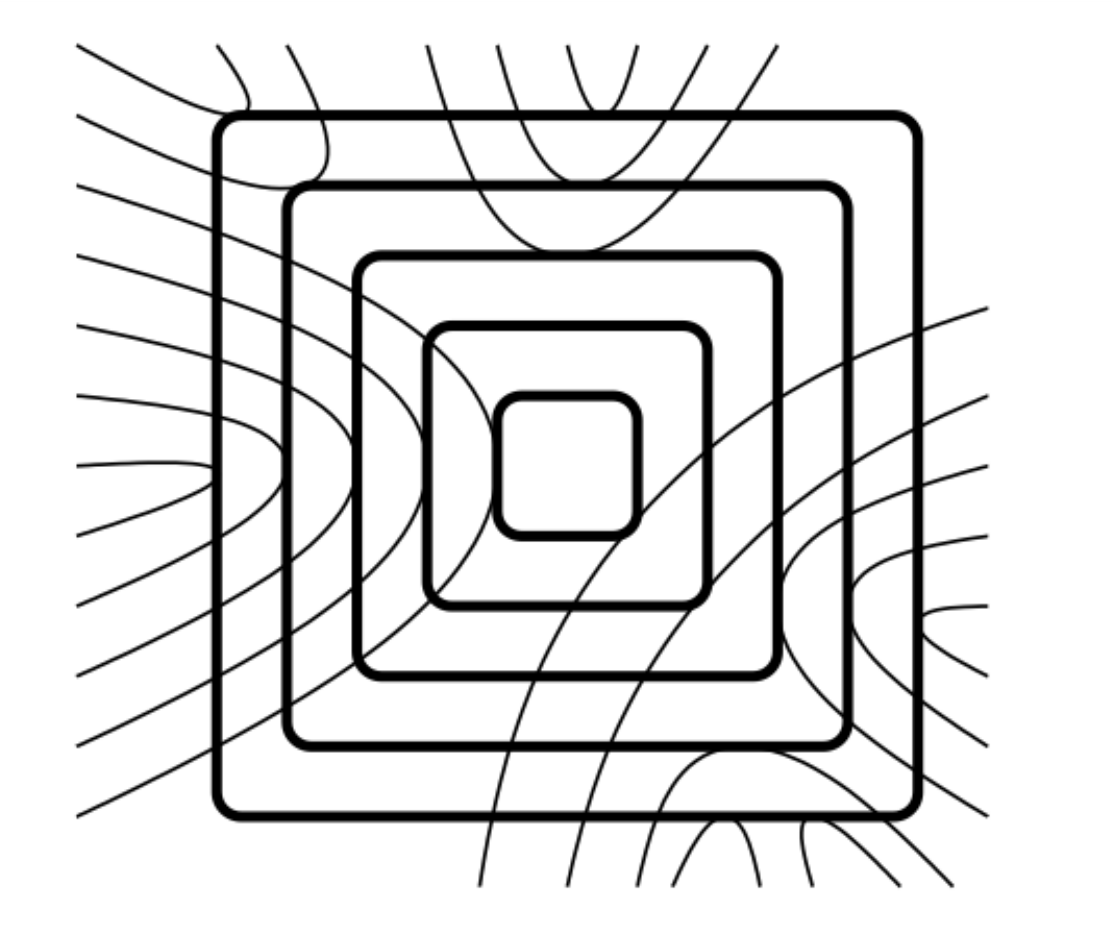}
\end{center}
\caption{An example of a convex CL-configuration $({\cal C},{L})$. In the picture,
only the 5 cycles in ${\cal C}$  and a cropped portion of $L$ is depicted.}
\label{fig:convex}
\end{figure}

\subsection{Convex configurations}

We introduce CL-configurations with particular 
characteristics that will be useful for the subsequent 
proofs. We then show that these characteristics are implied 
by tightness  and cheapness.

\paragraph{Convex CL-configurations.}
A segment $P$ of  ${\cal Q}$ is {\em convex}
if the following three conditions are satisfied:
\begin{itemize}
\item[(i)] it has no $0$-chord and 
\item[(ii)] for every $i\in\{1,\ldots,r\}$,  the following hold:
\begin{itemize}
\item[a.] $P$ has at most one $i$-chord 
\item[b.] if  $P$ has  an  $i$-chord, then $P\cap C_{i-1}\neq\varnothing$.
\item[c.] Each $i$-chord of $P$ has exactly two $i$-semichords.
\end{itemize}
\item[(iii)] If $P$ has eccentricity $i<r$, there 
is another segment inside the   zone of $P$
with eccentricity $i+1$.
\end{itemize}
We say ${\cal Q}$ is  {\em  convex}  if all its segments are convex.

\begin{observation}
\label{obs2}
Let ${\cal Q}=({\cal C},L)$ be a {CL-configuration} and let $(G^*,{\cal Q}^*=({\cal C}^*,L^*))$ be the reduced pair of $G$ and ${\cal Q}$. Then 
${\cal Q}$ is convex if and only ${\cal Q}^{*}$ is convex.
\end{observation}


%
%
%

%
%

%
%
%

%

%
%
%

\begin{figure}[ht]
\begin{center}
\includegraphics[width=5.6cm]{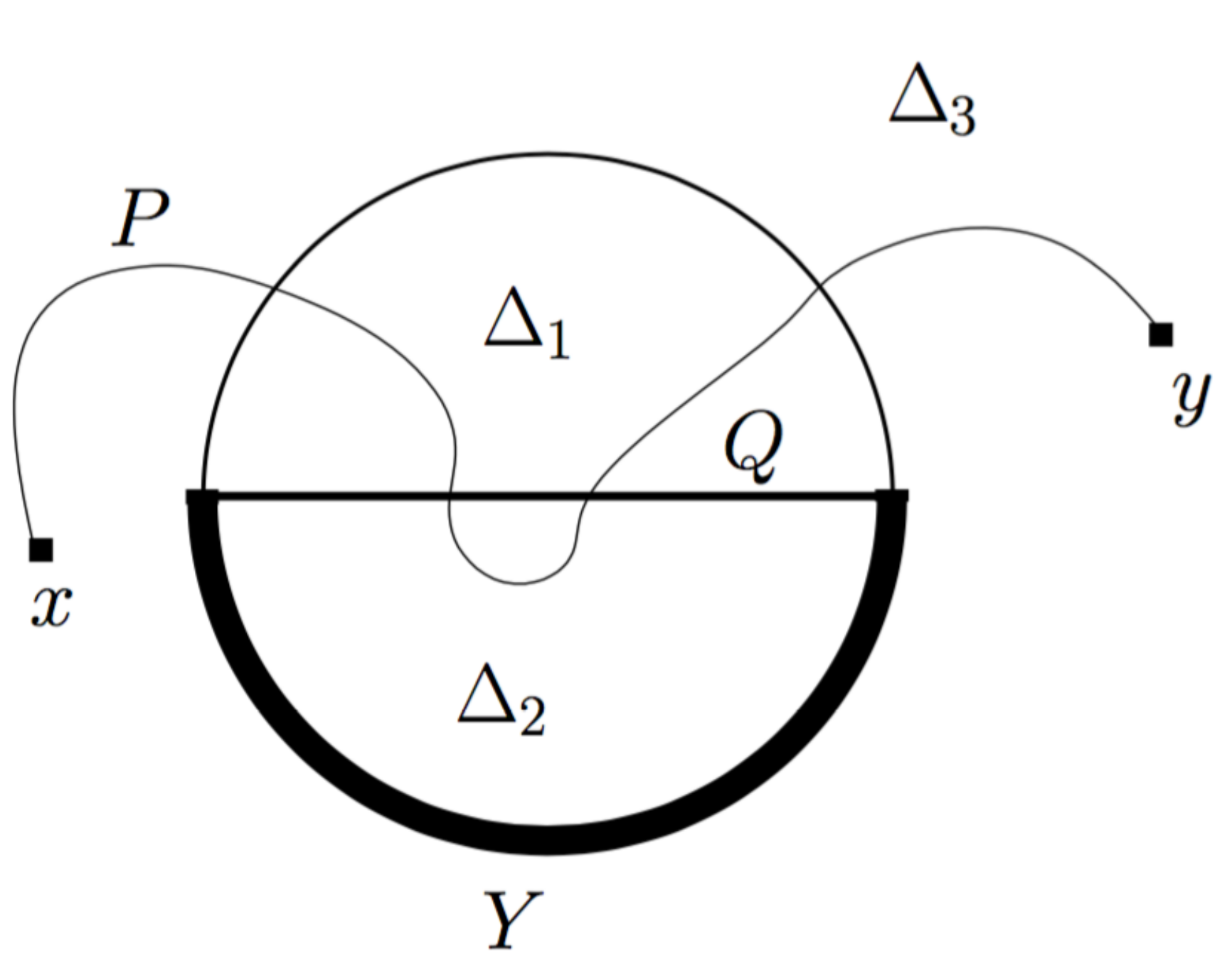}
\end{center}
\caption{A visualization of the conditions of Lemma~\ref{oi94jggdu3j}.}
\label{fig:prooftc}
\end{figure}

The proof of the  following lemma uses elementary topological arguments.

\begin{lemma}
\label{oi94jggdu3j}
Let $\Delta_{1},\Delta_{2}$ be closed disks of $\Bbb{R}^2$
where $\int(\Delta_{1})\cap\int(\Delta_{2})=\varnothing$ and 
such that $\Delta_{1}\cup \Delta_{2}$ is also a closed disk.
Let $\Delta_{3}=\Bbb{R}^{2}\setminus \int(\Delta_{1}\cup \Delta_{2})$ and let 
$Y=\bnd(\Delta_{3})\cap \Delta_{2}$ and $Q=\trim(\Delta_{1}\cap \Delta_{2})$.
Let $P$ be a closed arc of $\Bbb{R}^{2}$ whose endpoints 
are not in $\Delta_{1}\cup \Delta_{2}$ and such that
$Y\cap P=\varnothing$
and $Q\cap P\neq\varnothing$.
Then $\int(\Delta_{1})\cap P$ has at least two connected components.
\end{lemma}

 \begin{proof}
 Let $q$ be some point in $Q\cap P$.
 Let $Q'$ be an open arc that is a subset of $\int(\Delta_{1})$ and 
 has the same endpoints as $Y$. Notice that $q$ and $x$ belong to different open disks 
 defined by the cycle $Q'\cup Y$. Therefore $P$ should intersect $Q'$ 
or $Y$. As $Y\cap P=\varnothing$,  $P$ intersects $Q'$. As $Q' \subseteq \int(\Delta_{1})$,
 $\int(\Delta_{1})\cap P$ has at least one connected component.

Assume now that $\int(\Delta_{1})\cap P$ has exactly one connected component.
 Clearly, this  connected component will be an open arc $I$
 such that at least one of the endpoints of $I$, say $q$, belongs to $Q$.
 Moreover, there is a subset $P'$ of $P$ that is a closed arc where $P'\cap I=\varnothing$ 
 and whose endpoints are $q$ and one of $x$ and $y$, say $y$.
 As  $\int(\Delta_{1})\cap P$ has exactly one connected component,
 it holds that  $P'\cap \int(\Delta_{1})=\varnothing$. 
Let $Q'$ be an open arc that is a subset of $\int(\Delta_{1})$ and 
 has the same endpoints as $Y$. Notice that $q$ and $y$ belong to different open disks 
 defined by the cycle $Q'\cup Y$. Therefore $P'$ should intersect $\int(\Delta_{1})$ 
or $Y$, a contradiction as $P'\subseteq P$ and $Y\cap P=\varnothing$. 
 \end{proof}

%
%
%
%

\begin{lemma}
\label{DiscTheorem1}
Let $G$ be a plane graph and ${\cal Q}=({\cal C},L)$ be a {CL-configuration}
of $G$ where ${\cal C}$ is tight in $G$ and $L$ is ${\cal C}$-cheap. 
Then ${\cal Q}$ is  convex.
\end{lemma}

\begin{proof}
By Observations~\ref{obs1} and~\ref{obs2}, we may assume that ${\cal Q}$ is reduced.
Consider any segment of ${\cal Q}$. We show that it satisfies the three conditions
of  convexity. Conditions (i) and (ii).b follow directly from the
tightness of ${\cal C}.$ Condition (iii) follows from the fact that
$L$ is ${\cal C}$-cheap. 
In the rest of the proof we show Conditions (ii).a. and (ii).c.
For this, we consider the minimum $i\in\{0,\ldots,r\}$
such that one of these two conditions is violated.
From Condition  (i), $i\geq 1$.
Let $W$ be a segment of ${\cal Q}$ containing an $i$-chord $X$
for which one of Conditions (ii).a, (ii).c is violated.

%
%


\noindent We now define the set $Q$ according to which of the two conditions is violated.
We distinguish two cases:\medskip

\noindent{\em Case 1.} Condition (ii).c is violated. From  
Condition (ii).b, $X\setminus D_{i-1}$ contains 
more than two $i$-semichords of $X$.
Let $J_{1}$ be the biconnected outerplanar graph defined by the union of $C_{i-1}$
and the  $i$-semichords of $X$ that do not have an endpoint
on $C_{i}$. As there are at least three $i$-semichords in $X$,
$J_1$ has at least one internal edge
and therefore at least two simplicial  faces.
Moreover there are exactly two $i$-semichords of $X$, say $K_1$, $K_{2}$,
that have an endpoint in $C_{i}$ and $K_{1}$ and $K_{2}$  belong 
to the same, say $F'$, face of $J_{1}$.
Let $\Delta_{2}$ be the closure of a simplicial face of $J_{1}$ that is not $F'$.  \smallskip

\begin{figure}[ht]
\begin{center}
\includegraphics[width=13cm]{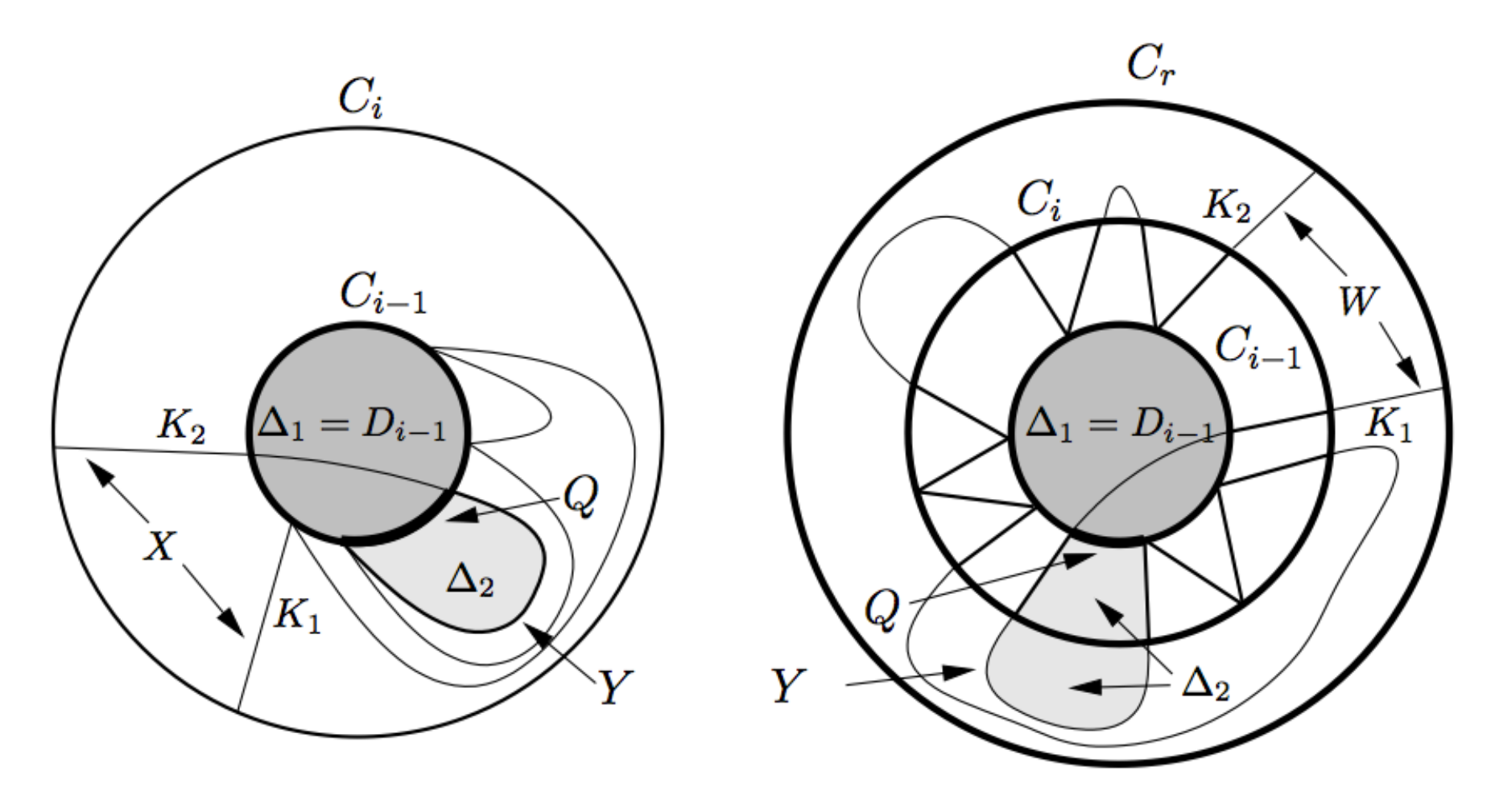}
\end{center}
\caption{The two cases of the proof of Lemma~\ref{DiscTheorem1}.
In the left part is depicted 
an $i$-chord $X$ that has 6 $i$-semichords
and in the right part is depicted a segment $W$ and the way it crosses the cycles $D_{i}$
and $D_{i-1}$. In the figure on the 
right the segment $W$ has $7$ $i$-chords and $14$ $i$-semichords.}
\label{fig:proofdt}
\end{figure}

\noindent{\em Case 2.}  Condition (ii).c holds while  Condition (ii).a is violated.
Let $J_{2}$ be the  biconnected outerplanar graph defined by the union 
of $C_{i-1}$ and the  connected components of $W\setminus D_{i-1}$
that do not contain  endpoints of $W$ in their boundary.
Notice that the rest of the connected components of $W\setminus D_{i-1}$
are exactly two, say $K_{1}$ and $K_{2}$. Notice that $K_{1}$ and $K_{2}$
are subsets of  the same face, say $F{'}$,
 of $J_{2}$.
As there are at least two $i$-chords in $W$, $J_{2}$
contains at least one internal edge and therefore at least two simplicial  faces.
Let $\Delta_{2}$ be the closure of a simplicial face of $J_{2}$ that is not $F'$.  \smallskip

%

In both of the above cases, we set $\Delta_{1}= D_{i-1}$, 
 $\Delta_3=\Bbb{R}^{2}\setminus \int(\Delta_{1}\cup \Delta_{2})$,  $Y=\bnd(\Delta_{3})\cap \Delta_{2}$,
and   $Q=\trim(\Delta_{1}\cap \Delta_{2})$.
Notice that $Y=\bnd(\Delta_{2})\setminus Q$,  therefore  $Y\subseteq W$.

We claim that  $L\cap Q\neq\varnothing$. Suppose not.
We consider $W'$ as the path in $W\cup Q$ that contains $Q$ as a subset and 
has the same endpoints  as $W$.
Then, $L'=(L\setminus W)\cup W'$ is a  linkage,  equivalent to $L$, where $c(L')<c(L)$, a contradiction to the fact that ${L}$ is ${\cal C}$-cheap. 
We just proved that $L\cap Q\neq\varnothing$ which in turn implies that $L$ contains a segment $P$ for which $P\cap Q\neq\varnothing$.
We distinguish two cases:\medskip

\noindent{\em Case A.} $W\ne P$.   This implies that $W\cap P=\varnothing$. As   $Y\subseteq W$, it follows that $Y\cap P=\varnothing$. Therefore, by Lemma~\ref{oi94jggdu3j}, 
$\int(\Delta_{1})\cap P$ has at least two connected components,
therefore $P$ has at least two $(i-1)$-chords. If $i>1$, 
then Condition (ii).a is violated for $i-1$, which contradicts the choice of $i$. If $i=1$, then $P$ has at least one $0$-chord, 
which violates Condition (i), that, as explained at the beginning
of the proof, holds for every segment of ${\cal Q}$.


\noindent{\em Case B.} $W= P$. Recall that $Y\subseteq W$ therefore $Y\subseteq P$. Let $p_{1}$ and $p_{2}$ be the endpoints 
of $Q$. As ${\cal Q}$ is reduced there exists two disjoint 
closed arcs $Z_{1}$ and $Z_{2}$ with endpoints $p_1,p_{1}'$
and $p_{2},p_{2}'$ respectively, such that 
\begin{itemize}
\item $p_{i}$ is an endpoint of $Z_{i}, i\in\{1,2\}$.
\item $Z_{i}\subseteq \clos(Q), i\in\{1,2\}$,
and 
\item $P\cap Z_{i}=\{p_{i}\}, i\in\{1,2\}$. 
\end{itemize}

Consider also a closed  arc $Y'$
that is a subset of $\int(\Delta_{2})\cup\{p_{1}',p_{2}'\}$ that does not intersect $L$ and whose endpoints are $p_{1}'$ and $p_{2}'$.
Let now $\Delta_{1}'=
\Delta_{1}$, let $\Delta_{2}'$ be the closed disk defined by 
the cycle $\clos(Q\setminus (Z_{1}\cup Z_{2}))\cup Y'$
that is a subset of $\Delta_{2}$. Let also $\Delta_{3}'=\Bbb{R}^{2}\setminus \int(\Delta_{1}'\cup \Delta_{2}')$ and $Q'=\trim(\Delta_{1}'\cap \Delta_{2}')$. 
As $Y'$ does not intersect $L$, we obtain $Y'\cap P=\varnothing$.
Observe that $Z_{1},Q',Z_{2}$ form a partition of $Q$. As $Q\cap P\neq\varnothing$ and $(Z_{i}\setminus\{p_i\})\cap P=\varnothing, i\in\{1,2\}$, we conclude that $Q'\cap P\neq\varnothing$.

By applying Lemma~\ref{oi94jggdu3j}, $\int(\Delta_{1}')\cap P$ has at least two connected components. Therefore $P$ has at least two $(i-1)$-chords. This yields a contradiction, as in Case A.
\end{proof}

\subsection{Bounding the number of extremal segments}

In this subsection we prove that the number of extremal segments 
is  bounded by a linear function of the  number of linkage paths.

\paragraph{Out-segments, hairs, and flying hairs.}
 Let $G$ be a plane graph and ${\cal Q}=({\cal C},L)$ be a {CL-configuration}
of $G$ of depth $r$.
An {\em out-segment} of $L$ is a subpath $P'$ of a path in ${\cal P}(L)$ such that 
the endpoints of $P'$ are in $C_{r}$ and the internal vertices of $P'$ are not in $D_{r}.$ 
A {\em hair} of $L$ is a subpath $P'$ of a path in ${\cal P}(L)$ such that 
one endpoint of $P'$ is in  $C_{r}$, the other is a terminal of $L,$  
and the internal vertices of $P'$ are not in $D_{r}.$
A {\em flying hair} of $L$ is a path in ${\cal P}(L)$ that does not intersect $C_{r}.$

Given a linkage $L$ of $G$ and a closed disk $D$ of $\Bbb{R}^{2}$ whose boundary is a cycle of $G$, we define 
${\bf out}_{D}(L)$ to be the graph obtained from the graph $(L\cup \bnd(D))\setminus\intr(D)$ after dissolving all vertices of degree 2.
For example ${\bf out}_{D_{r}}(L)$ is a plane graph consisting of the out-segments, the hairs,   the flying hairs of $L$, and what results 
from  $C_{r}$ after dissolving its vertices of degree 2  that do not belong in $L$.
Let $f$ be a face of ${\bf out}_{D_{r}}(L)$ that is different from $\intr(D_{r}).$ We say that  $f$ is a {\em cave of} ${\bf out}_{D_{r}}(L)$ if 
the union of the out-segments and extremal segments in the boundary of $f$ is a connected set. Recall that a segment of  ${\cal Q}$ is {extremal} 
if it is has  eccentricity $r$, i.e., it is a subpath of $C_{r}.$

Given a plane graph $G$, we say that two edges $e_{1}$ and $e_{1}$ are {\em cyclically adjacent} if they have a common endpoint $x$ and appear consecutively in the cyclic
ordering of the edges incident to $x$, as defined by the embedding of $G$.
A subset $E$  of $E(G)$ is {\em cyclically connected} if for every two edges $e$ and $e'$ in $E$ 
there exists a sequence of edges $e_{1},\ldots,e_{r}\in E$ where $e_{1}=e$, $e_{r}=e'$ and for 
each $i\in\{1,\ldots,r-1\}$\ $e_{i}$ and $e_{i+1}$ are cyclically adjacent.

Let  ${\cal Q}=({\cal C},L)$ be a {CL-configuration}. We say that ${\cal Q}$
is {\em touch-free} if for every path $P$ of $L$, the number of the connected components
of $P\cap C_{r}$ is not $1$.

\begin{lemma}
\label{boundsegextr}
Let $G$ be a plane graph and ${\cal Q}=({\cal C},L)$ be a touch-free {CL-configuration}
of $G$ where ${\cal C}$ is tight in $G$ and $L$ is ${\cal C}$-cheap. The number of extremal segments of ${\cal Q}$ is at most $2\cdot|{\cal P}(L)|-2$.
\end{lemma}

\begin{proof}
Let $(G^*,{\cal Q}^*=({\cal C}^*,L^*))$ be the reduced pair of $G$ and ${\cal Q}$.
Notice that, by Observation~\ref{obs1},  
${\cal C}^*$ is tight in $G$ and $L^*$ is ${\cal C}^*$-cheap.
Moreover, it is easy to see that ${\cal Q}^*$ is touch-free and ${\cal Q}$ and ${\cal Q}^{*}$ have the same number of extremal segments which are all trivial paths (i.e., paths consisting of  only one vertex).
Therefore, it is sufficient to prove that the lemma holds for ${\cal Q}^{*}$. 
Let $\rho$ be the number of extremal segments of ${\cal Q}^*$.

Let $J={\bf out}_{D^{*}_{r}}(L^*)$ and 
$k=|{\cal P}(L^*)|.$  Notice that the number of extremal segments of ${\cal Q}^*$ is equal to the number of vertices of degree 4 in $J$.
%
%
%

%
%
%
%
%
The terminals of $L^*$ are partitioned in three families
\begin{itemize}
\item {\em flying} terminals, $T_{0}$: endpoints of flying hairs.
\item {\em invading} terminals $T_{1}$: these are endpoints of hairs whose non terminal endpoint has degree 3 in  $J$.
\item {\em  bouncing} terminals $T_{2}$:  these are endpoints of hairs whose non terminal endpoint has degree 4 in $J$.
\end{itemize}
\medskip

\noindent A hair containing an invading and bouncing terminal is  called {\em invading}  and {\em bouncing hair} respectively.

Recall that $|T_{0}|+|T_{1}|+|T_{2}|= 2k.$\smallskip

\noindent{\em Claim 1.} The number of caves of $J$ is at most the number of invading terminals. \smallskip

\noindent{\em Proof of claim 1.} 
Clearly, a hair cannot be in the common boundary of two caves. Therefore it is enough to prove 
that the set obtained by the union of a cave $f$ and its boundary  contains at least one invading hair. Suppose this is not true.  
Consider the open arc $R$  obtained if we remove from $\bnd(f)$ all the points that belong to 
out-segments. 
Clearly, $R$ results from a subpath $R^+$ of $C_{r}^*$ 
after removing its endpoints, i.e., $R=\trim(R^{+})$. 

Notice that because $f$ is a cave,  $R$ is a non-empty connected subset of $C_{r}^*$. Moreover, $R\cap L^*$ is non-empty, otherwise 
$L^{*\prime}=(L^*\setminus(\bnd(f))\cup R$ is also a linkage with the same pattern as $L^*$
where $c(L^{*\prime})<c(L^*),$ a contradiction to the fact that  $L^*$ is ${\cal C}^*$-cheap.
Let $Y$ be a connected component of $R\cap L^*$. As ${\cal Q}^{*}$ is reduced, $Y$
consists of a single vertex $y$ in the  open set $R$. Notice that $Y$ is a subpath of 
a segment $Y'$ of ${\cal Q}^*$. We claim that $Y'$ is not extremal. 
Suppose to the contrary that  $Y'$ is extremal.  Then $Y'=Y$ and there should be two distinct out-segments that have $y$ as a common endpoint. This contradicts the fact that $y\in R$.

By Lemma~\ref{DiscTheorem1}, ${\cal Q}^*$ is convex, therefore one of the endpoints of the non-extremal segment $Y'$ is $y$ and thus is  in $R$ as well. 
This means that $y$ is the endpoint of one out-segment which
again contradicts the fact that $y\in R$.
%
%
%
This completes the proof of  Claim~1.\medskip

%

\begin{figure}[ht]
\begin{center}
\includegraphics[width=14.6cm]{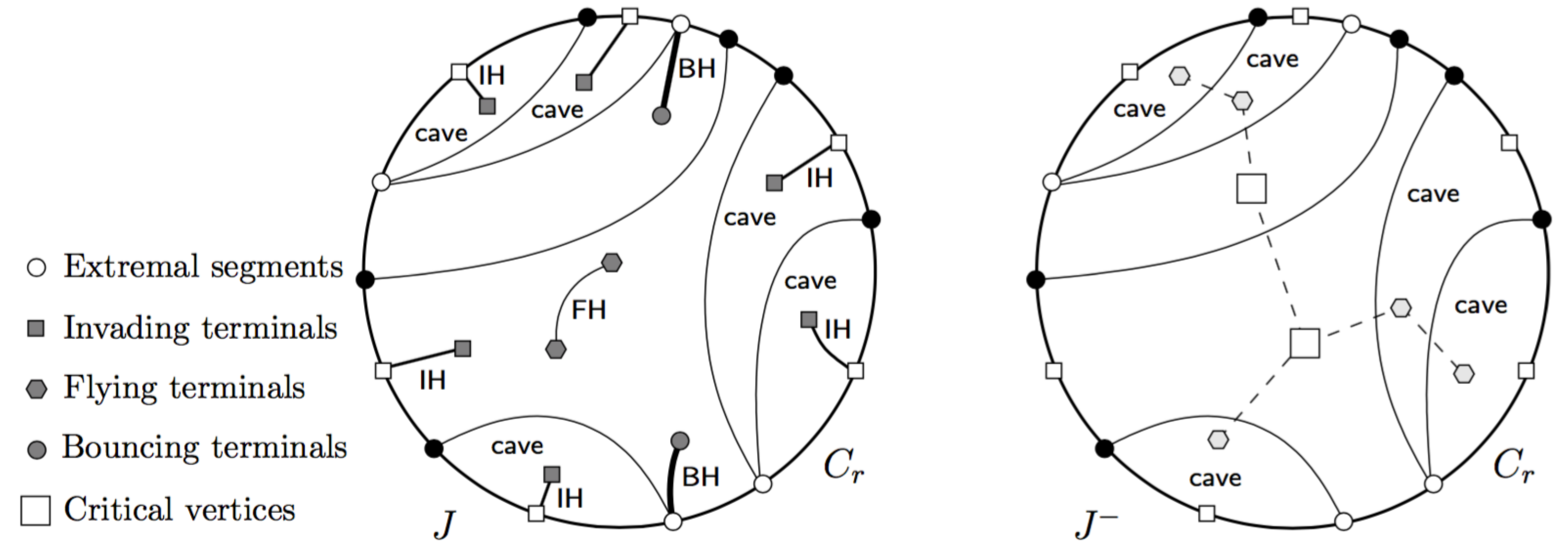}
\end{center}
\caption{Examples of the graphs $J$ and $J^-$
in the proof of Lemma~\ref{boundsegextr} (the outer face in the picture
corresponds to the interior of $D_{r}$). The faces that are 
caves contain the word {\sf\small cave}. {\sf\small FH}: flying hair, {\sf\small BH}: bouncing hair, {\sf\small IH}: invading hair.
The forest $\Psi=J^-\setminus E(C_{r})$ has 6 edges and 4 connected  components.
The weak dual $T$ of $J^{-}$ is depicted with dashed lines. The large white square 
vertices are the rich vertices of  $T$.}
\label{fig:prsosoftsc}
\end{figure}

Let $J^-$ be the graph obtained from $J$ by removing all hairs
and notice that $J^-$ is a biconnected outerplanar graph.
Let $S$ be the set of vertices  of $J^{-}$ that have degree $4$.
Notice that, because ${\cal Q}^*$ is touch-free, $|S|$ is
equal to the number of vertices of $J$ that have degree $4$ minus the 
number of bouncing terminals. Therefore,
\begin{eqnarray}
\rho=|T_{2}|+|S|. \label{tlllokkd}
\end{eqnarray}
Notice that  if we remove from $J^{-}$ all the edges of $C_{r}^*,$ the resulting 
graph is a forest $\Psi$ whose connected components are paths.  Observe that none of these paths is a trivial path because ${\cal Q}^*$
is {touch-free}.
We denote by $\kappa(\Psi)$
the number of connected components of $\Psi$.
Let $F$ be the set of faces of $J^-$ that are different from $D_{r}^*$.
$F$ is  partitioned  
into the faces that are caves, namely $F_{1}$ 
and the non-cave faces, namely $F_{0}.$
By the Claim 1,  $|F_1|\leq |T_{1}|.$

To complete the proof, it  is enough to show that 
\begin{eqnarray}
|S| & \leq & |T_{1}|-2 \label{thsisosnge}
\end{eqnarray}
Indeed the truth of~\eqref{thsisosnge} along with~\eqref{tlllokkd}, would imply 
that $\rho$ is at 
most $|T_{2}|+|S|\leq |T_{2}|+|T_{1}|-2\leq |T|-2=2k-2.$\medskip

We now return to the proof of~\eqref{thsisosnge}. For this, we need two more claims.\medskip

\noindent{\em Claim 2}: $|F_{0}|\leq \kappa(\Psi)-1$.

\noindent{\em Proof.} We use  induction on $\kappa(\Psi).$ Let $K_{1},\ldots,K_{\kappa(\Psi)}$
be the connected components of $\Psi$.
If $\kappa(\Psi)=1$ then all faces in $F$ are caves, therefore $|F_0|=0$ and we are done.
Assume now that $\Psi$ contains at least two connected components.

We assert that there exists at least one connected component $K_{h}$ of $\Psi$
with the property that only one non-cave face of $J^-$ contains edges of $K_{h}$ in its boundary.
To see this, consider the weak dual  $T$ of $J^-$. Recall that, as $J^-$ is biconnected, $T$ is a tree.
Let  $K_{i}^{*}$ be the subtree of $T$ containing the duals of the edges in $E(K_{i})$, $i\in\{1,\ldots,\kappa(\Psi)\},$ and observe that 
$E(K_{1}^{*}),\ldots,E(K_{\kappa(\Psi)}^{*})$ 
is a partition of $E(T)$ 
into $\kappa(\Psi)$ cyclically connected sets.
We say that a vertex of $T$ is {\em rich} if it is incident 
with edges in more than one members of $\{K_{1}^*,\ldots,K_{\kappa(\Psi)}^*\}$, otherwise 
it is called {\em poor} (see Figure~\ref{fig:prsosoftsc}). 
Notice that a vertex of $T$ is rich if and only if its dual face in $J^{-}$
is a non-cave.
We call a subtree  $K_{i}^{*}$ {\em peripheral} if $V(K_{i}^{*})$  contains at most one rich vertex of $T$.
Notice that the claimed property for a component in $\{K_{1},\ldots,K_{\kappa(\Psi)}\}$
is equivalent to the existence of a peripheral subtree in  $\{K_{1}^*,\ldots,K_{\kappa(\Psi)}^*\}$.
To prove that such a peripheral subtree exists, consider a path $P$ in $T$ intersecting the vertex sets of 
a maximum number of  members of $\{K_{1}^*,\ldots,K_{\kappa(\Psi)}^*\}$.
Let $e^*$ be the first edge of $P$ and let 
$K_{h}^{*}$ be the unique subtree whose edge set  contains $e^*$. Because of the maximality of 
the choice of $P$, $V(K_{h}^{*})$ contains exactly 
one rich vertex $v_{h}$, therefore $K_{h}^{*}$ is peripheral and the assertion follows.
We denote by $f_{h}$ the non-cave face of $J^{-}$ that is the dual of $v_{h}$.\medskip


Let $H^-$ be the  outerplanar graph obtained from $J^-$
after removing the edges of $K_{h}$. Notice that this removal 
results in the unification of all faces that are incident to the 
edges of $K_{h},$ including $f_{h}$, to  a single face $f^+$.
By the inductive hypothesis 
the number of non-cave faces of $H^{-}$ is at most $\kappa(\Psi)-2$.
Adding back the edges of $K_{h}$ in $J^-$ restores $f_{h}$ 
as a distinct non-cave face of $J^-$. If $f^+$ was a non-cave of $H^-$
then $|F_{0}|$ is equal to the number of non-cave faces of $H^-$, else
$|F_{0}|$ is one more than this number. In any case,
$|F_{0}|\leq \kappa(\Psi)-1$, and the claim follows.\medskip

\noindent{\em Claim 3}:  $|V(\Psi)|\leq |T_{1}|+2 \cdot \kappa(\Psi)-2.$ 

\noindent{\em Proof.}
Let $T$ be the weak dual of $J^-$.
Observe that $ |F_{0}|+|F_{1}|=|F|=|V(T)|=|E(T)|+1=|E(\Psi)|+1=|V(\Psi)|- \kappa(\Psi)+1.$
Therefore $|V(\Psi)|= |F_{0}|+|F_{1}|+ \kappa(\Psi)-1.$ Recall that, by Claim 1, $|F_{1}|\leq |T_{1}|$
and, taking into account Claim 2, we conclude that  $|V(\Psi)|\leq  |T_{1}|+2 \cdot \kappa(\Psi)-2$. Claim 3 follows.\medskip

Notice now that a vertex of $J^-$ has degree 4 iff it is an internal vertex of some path in $\Psi$.
Therefore,  as all connected components of $\Psi$ are non-trivial paths, it holds that $|V(\Psi)|=|S|+|L(\Psi)|= |S|+2\cdot \kappa(\Psi)$, where  $L(\Psi)$ is the 
set of leaves of $\Psi$.
 By Claim 3,
$$|S|+2\cdot \kappa(\Psi)= |V(\Psi)|\leq |T_{1}|+2 \cdot \kappa(\Psi)-2 \Rightarrow |S|\leq |T_{1}|-2.$$
Therefore,~\eqref{thsisosnge} holds and this completes the proof of the lemma.
\end{proof}

\subsection{Bounding the number and size of segment types}
\label{subsec:bounds}

In this section we introduce the notion of segment type
that partitions the segments into classes of mutually ``parallel''
segments.  We next prove that, in the light of the results of the 
previous section, the number of these classes is bounded by a 
linear function of the number $k$ of linkage paths.
In Subsections~\ref{k7iope4} and~\ref{llt38f}
we show that if one of these equivalence classes 
has size more than $2^{k}$, then an   equivalent 
cheaper linkage can be found. All these facts will be 
employed in the culminating Subsection~\ref{subsec:irre}
in order to prove that a cheap linkage 
cannot go very ``deep'' into the cycles of a cheap CL-configuration. 
That way we will be able to quantify the depth at which 
an irrelevant vertex is guaranteed to exist.

\paragraph{Types of segments.}
Let $G$ be a plane graph and 
let ${\cal Q}=({\cal C},L)$ be a  convex CL-configuration of 
$G$. 
Let $S_{1}$, $S_{2}$  be two segments of ${\cal Q}$
and let $P$ and $P'$ be the two paths on $C_r$
connecting an endpoint of $S_{1}$ with an endpoint of $S_{2}$ 
and passing through no other endpoint of $S_1$ or $S_2.$
We say that $S_{1}$ and $S_{2}$ are {\em parallel}, and 
we write $S_1\parallel S_2$, if
\begin{enumerate}
\item[(1)] no segment of ${\cal Q}$ has both endpoints on $P.$
\item[(2)] no  segment of ${\cal Q}$ has both
endpoints on $P'.$ 
\item[(3)] the closed-interior of  the cycle $P\cup S_{1}\cup P'\cup S_{2}$ does not contain
the disk $D_{0}$.
\end{enumerate} 
A \emph{type of segment} is an equivalence class of segments of ${\cal Q}$ under the relation $\parallel.$ 
%
%

%

Given a linkage $L$ of $G$ and a closed disk $D$ of $\Bbb{R}^{2}$ whose boundary is a cycle of $G$, we define 
${\bf in}_{D}(L)$ to be the graph obtained from $(L\cup\bnd(D))\cap D$ after dissolving all vertices of degree 2.

Notice  that ${\bf in}_{D_{r}}(L)$
is the biconnected outerplanar graph  formed if we dissolve 
all vertices of degree 2 in the graph that is formed by the union of $C_{r}$ and the 
segments of ${\cal Q}$. As ${\cal Q}$ is convex,
one of the faces of ${\bf in}_{D_{r}}(L)$ contains the interior of $D_{0}$ and we call this face {\em central} face.
We define the {\em segment tree} of ${\cal Q}$, denoted by $T({\cal Q}),$ 
as follows. 
\begin{itemize}
\item Let $T^-$ be the weak dual of  ${\bf in}_{D_{r}}(L)$ 
rooted
at the vertex that is the
dual of the central face.
\item Let $Q$ be the set of leaves of $T^-$.
For each vertex $l\in Q$ do the following:
Notice first that $l$  is the dual of a face $l^*$ of ${\bf in}_{D_{r}}(L)$.
Let $W_{1},\ldots,W_{\rho_{l}}$ be the extremal segments in the boundary of $l^*$ (notice that, by the convexity of ${\cal Q}$, for every $l$,\ $\rho_{l}\geq 1$).
Then, for each $i\in\{1,\ldots,\rho_{l}\}$, create a new leaf $w_{i}$ corresponding to the extremal segment $W_{i}$ and make it adjacent to $l$.
\end{itemize}
The {\em height} of $T({\cal Q})$ is the maximum distance 
from its root to its leaves. The {\em real height}
of $T({\cal Q})$ is the maximum number of internal vertices of degree at least $3$
in a path from its root to its leaves plus one. The {\em dilation}
of $T({\cal Q})$ is the maximum length of a path all whose internal vertices 
have degree 2 and are different from the root.

\begin{observation}
\label{ogsdll5dfk}
Let $G$ be a plane graph and 
let ${\cal Q}=({\cal C},L)$ be a  convex CL-configuration of 
$G$.  Then the dilation of $T({\cal Q})$ is equal to the maximum cardinality of an equivalence class of $||$.
\end{observation}

\begin{observation}
\label{ogll5fk}
Let $G$ be a plane graph and 
let ${\cal Q}=({\cal C},L)$ be a  convex CL-configuration of 
$G$.  Then the height of $T({\cal Q})$ is upper bounded by the  dilation of $T({\cal Q})$
multiplied by  the real height of $T({\cal Q})$.
\end{observation}
%
%
%

\begin{figure}[ht]
\begin{center}
\includegraphics[width=12.5cm]{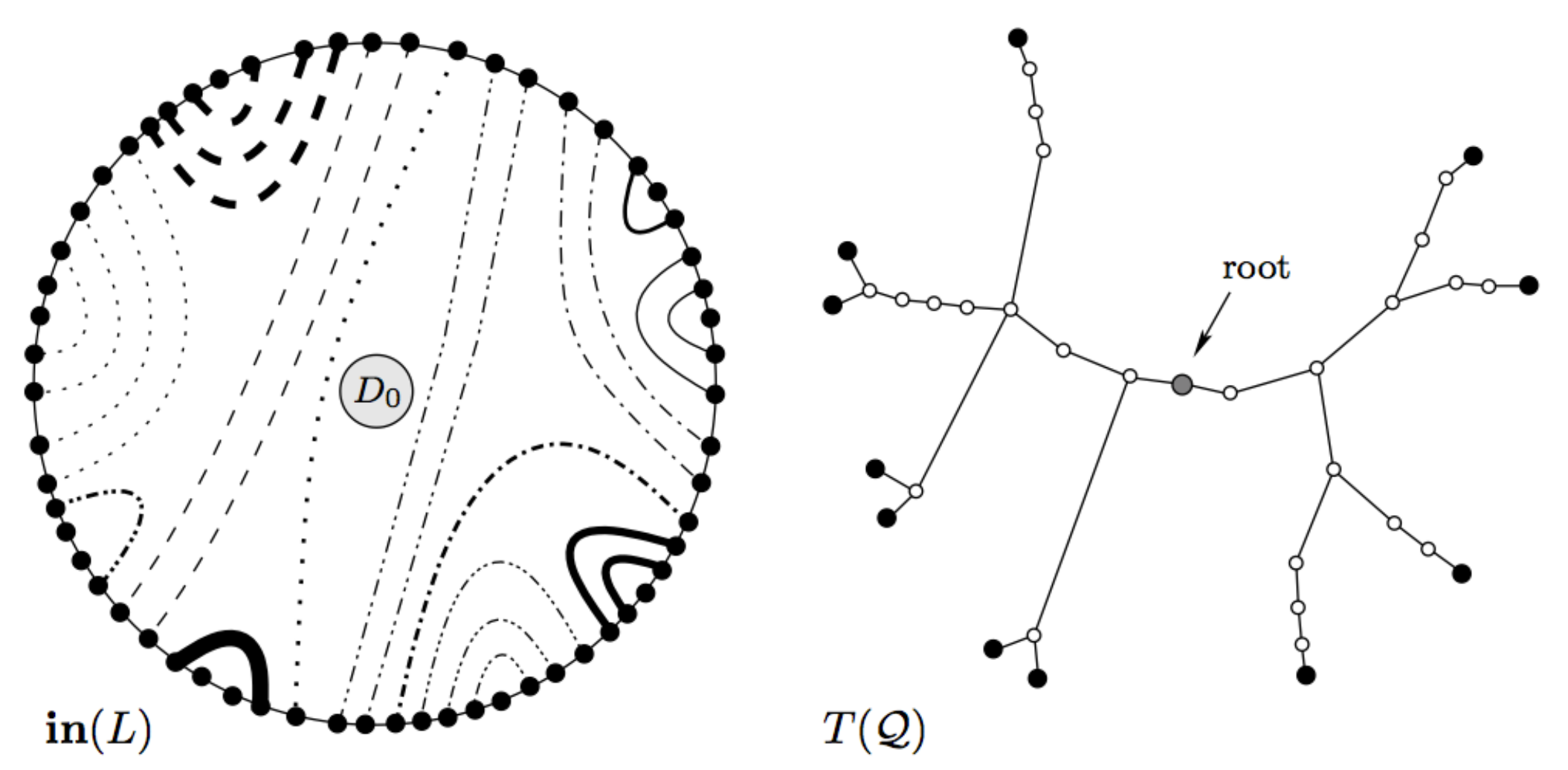}
\end{center}
\caption{The graph ${\bf in}_{D_{r}}(L)$ for some convex CL-configuration ${\cal Q}=({\cal C},L)$
and the tree $T({\cal Q})$. Internal edges in ${\bf in}_{D_{r}}(L)$ of the same type are drawn as lines of the same type. ${\cal Q}$ has 11 extremal segments, as many as the leaves of $T({\cal Q})$. The relation $\parallel$ has 19 equivalent classes.
The dilation of $T({\cal Q})$ is 4, its height is 8 and its real height is 4.}
\label{fig:inL}
\end{figure}

%

The following lemma is an immediate consequence of Lemma~\ref{boundsegextr}
and the definition of a segment tree.
The condition that  $L\cap C_{r}\neq\varnothing$ simply requires that the CL-configuration that we consider is non-trivial in the sense that the linkage $L$ enters the closed disk $D_{r}$.

\begin{lemma}\label{theo:bounding_number_of_types}
Let $G$ be a plane graph and ${\cal Q}=({\cal C},L)$ be a touch-free {CL-configuration}
of $G$ where ${\cal C}$ is tight in $G$, $L$ is ${\cal C}$-cheap, and $L\cap C_{r}\neq\varnothing$. 
Then ${\cal Q}$ is convex and the real height of the segment tree $T({\cal Q})$
is at most $2\cdot |{\cal P}(L)|-3.$
\end{lemma}

\begin{proof}
Certainly, the convexity of ${\cal Q}$ follows directly from Lemma~\ref{DiscTheorem1}. 
We examine the non-trivial case where $T({\cal Q})$ contains at least one edge.
We first claim that  $|{\cal P}(L)|\geq 2$. 
Assume to the contrary that $L$ consists of a single path $P$. As ${\cal Q}$ is convex
and $L\cap C_{r}\neq\varnothing$, 
${\cal Q}$  has at least one extremal segment. Suppose now that ${\cal Q}$  has more than one extremal segment all of which  are connected components  of $C_{r}\cap P$. 
Let $P_{1}$ and $P_{2}$
be the closures of the connected components of 
$L\setminus D_{r}$  that contain the terminals of $P$. Let $p_{i}\in V(C_{r})$ be the endpoint of $P_{i}$ that is not a terminal, $i\in\{1,2\}$. Let also $P'$ be any path in $C_{r}$ between $p_{1}$ and $p_{2}$. Notice now that $P_{1}\cup P'\cup P_{2}$ is a cheaper 
linkage with the same 
pattern as $L$, a contradiction to the fact that $L$ is ${\cal C}$-cheap. Therefore we conclude that ${\cal Q}$ has exactly one extremal segment, which contradicts  the fact that ${\cal Q}$ is touch-free.   This completes the proof that  $|{\cal P}(L)|\geq 2$.

Recall that, by the construction of $T({\cal Q})$ there is a 1--1 correspondence between the leaves of $T({\cal Q})$
and the extremal segments of ${\cal Q}$.
From Lemma~\ref{boundsegextr}, $T({\cal Q})$
has at most $2\cdot |{\cal P}(L)|-2$ leaves.
Also $T({\cal Q})$ has at least $2$ leaves, because ${\cal Q}$ is touch-free.
It is known that the number of internal vertices of degree $\geq 3$ in a tree with 
$r\geq 2$ leaves is at most $r-2$. Therefore,  $T({\cal Q})$ has at most 
$2\cdot |{\cal P}(L)|-4$ internal vertices of degree $\geq 3$.
Therefore the real height  of $T({\cal Q})$ is at most $2\cdot |{\cal P}(L)|-3$.
%
\end{proof}

\subsection{Tidy grids in convex configurations}
\label{k7iope4}

In this subsection we prove that the existence of many ``parallel'' segments 
implies the existence of a big enough grid-like structure.

\paragraph{Topological minors.}
We say that a graph $H$ is a {\em topological minor} of a graph $G$ 
if there exists an injective function $\phi_{0}: V(H)\rightarrow V(G)$ 
and a function $\phi_{1}$ mapping the edges of $H$ to paths of $G$
such that 
\begin{itemize}
\item 
for every edge $\{x,y\}\in E(H)$, $\phi_{1}(\{x,y\})$ is a path
between $\phi_0(x)$ and $\phi_0(y)$.
\item if two paths in $\phi_{1}(E(H))$ have a common vertex, then this vertex should 
be an endpoint of both paths.
\end{itemize}
Given the pair $(\phi_0,\phi_1)$, we say that $H$ is a {\em topological 
minor of $G$ via} $(\phi_0,\phi_1)$.

\paragraph{Tilted grids and $L$-tidy grids.}
Let $G$ be a graph. A {\em tilted grid} of $G$ 
is a pair ${\cal U}=({\cal X},{\cal Z})$ where ${\cal X}=\{X_{1},\ldots,X_{r}\}$ 
and ${\cal Z}=\{Z_{1},\ldots,Z_{r}\}$ 
are both collections of $r$ vertex-disjoint paths of $G$ such that 
\begin{itemize}
\item for each $i,j\in\{1,\ldots,r\}$ $I_{i,j}=X_{i}\cap Z_{j}$ is a (possibly 
edgeless) path of 
$G$,
\item for $i\in\{1,\ldots,r\}$ the subpaths $I_{i,1},I_{i,2},\ldots,I_{i,r}$ appear 
in this order in $X_{i}$.
\item for $j\in\{1,\ldots,r\}$ the subpaths $I_{1,j},I_{2,j},\ldots,I_{r,j}$ appear 
in this order in $Z_{j}$.
\item $E(I_{1,1})=E(I_{1,r})=E(I_{r,r})=E(I_{r,1})=\varnothing$,
\item Let $$G_{{\cal U}}=(\bigcup_{i\in\{1,\ldots,r\}}X_{i})\cup (\bigcup_{i\in\{1,\ldots,r\}}Z_{i})$$
and let  $G_{{\cal U}}^*$ be the graph taken from the graph
  after contracting all edges in $\bigcup_{(i,j)\in\{1,\ldots,r\}^{2}}I_{i,j}$.
 Then $G_{{\cal U}}^*$  contains the $(r\times r)$-grid $\Gamma$
 as a topological minor via a pair $(\chi_{0},\chi_{1})$ such that 
\begin{itemize}
\item[A.] the upper left (resp. upper right, down right, down left) corner 
of $\Gamma$ is mapped via $\chi_{0}$ to the  (single) endpoint
of $I_{1,1}$  (resp. $I_{1,r},I_{r,r},$ and $I_{r,1}$).
\item[B.] $\bigcup_{e\in E(\Gamma)}\chi_{1}(e)=G_{{\cal U}}^*$ (this makes $G_{{\cal U}}^*$ to be a subdivision of $\Gamma$).
\end{itemize}
\end{itemize}
We call the subgraph $G_{{\cal U}}$  of $G$ {\em realization} of the tilted grid ${\cal U}$
and the graph $G_{\cal U}^{*}$ {\em representation} of ${\cal U}$. We treat both $G_{\cal U}$
and $G_{\cal U}^{*}$ as plane graphs.
We also refer to the cardinality $r$ of ${\cal X}$ (or ${\cal  Z}$) as the {\em capacity} of ${\cal U}$.
The {\em  perimeter} of $G_{\cal U}$ is the cycle $X_{1}\cup Z_{1}\cup X_{r}\cup Z_{r}$.
Given a graph $G$ and a linkage $L$ of $G$ we say that 
a tilted grid ${\cal U}=({\cal X},{\cal Z})$ of $G$ is an {\em $L$-tidy tilted grid} of $G$ if
$D_{\cal U}\cap L=\cupall {\cal Z}$ where $D_{\cal U}$ is the closed-interior of the perimeter of $G_{\cal U}$.

\begin{figure}[ht]
\begin{center}
\includegraphics[width=10.1cm]{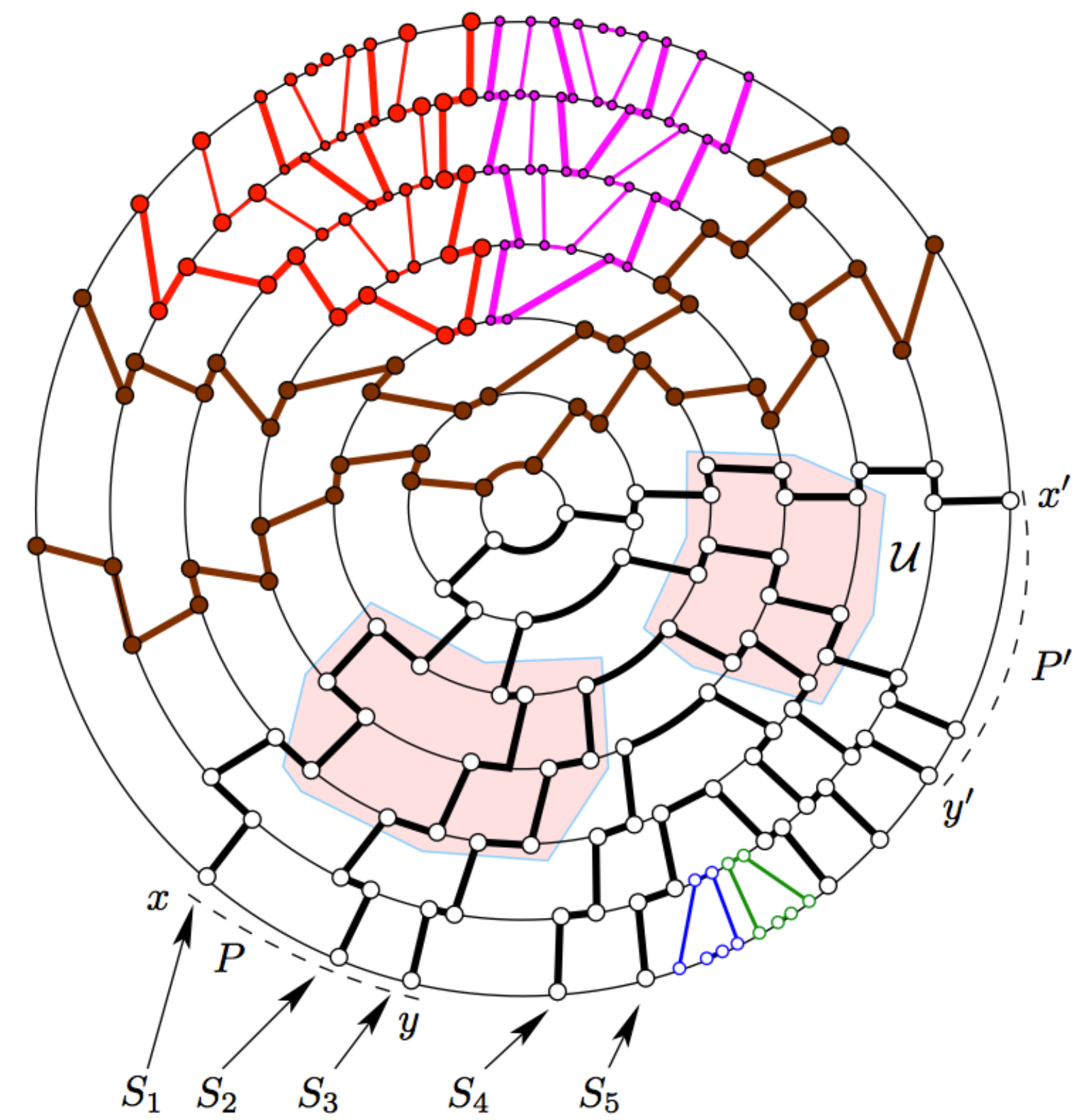}
\end{center}
\caption{A visualisation of the proof of Lemma~\ref{theo:bounding_more}.
It holds that 
$\sigma_{1}=0$ and $\sigma_{5}=4$, $m=5$, $m'=3$.
The two shadowed regions indicate the two connencted components 
of $D_{S}\cap A_{C}$.}
\label{fsig:thiss2sones}
\end{figure}

\begin{lemma}\label{theo:bounding_more}Let $G$ be a plane graph and 
let ${\cal Q}=({\cal C},L)$ be a  convex CL-configuration of 
$G$. 
Let also ${\cal S}$ be an equivalence class of the relation $\parallel$.
Then $G$ contains a tilted grid ${\cal U}=({\cal X},{\cal Z})$ of capacity $\lceil |{\cal S}|/2\rceil$
that is an $L$-tidy tilted grid of $G$.
\end{lemma}

\begin{proof}
Let ${\cal C}=\{C_{0},\ldots,C_r\}$ and let ${\cal S}=\{S_{1},\ldots,S_{m}\}$.
For each $i\in\{1,\ldots,m\}$, 
let $\sigma_{i}$ be the eccentricity of $S_{i}$
and let $\sigma^{\rm max}=\max\{\sigma_{i}\mid i\in\{1,\ldots,m\}\}$
and $\sigma^{\rm min}=\min\{\sigma_{i}\mid i\in\{1,\ldots,m\}\}$.
Convexity allows us to assume that 
$S_{1},\ldots,S_{m}$ are ordered in a way that 
\begin{itemize}
\item $\sigma_{1}=\sigma^{\rm min}$, 
\item $\sigma_{m}=\sigma^{\rm max}$, and
\item for all $i\in\{1,\ldots,m-1\}$, $\sigma_{i+1}=\sigma_{i}+1$.
\item for all $i\in\{1,\ldots,m\}$, $I_{i,\sigma_{i}}=S_{i}\cap C_{\sigma_{i}}$ is a subpath of $C_{\sigma_{i}}$.
\end{itemize}
Let $m'=\lceil\frac{m}{2}\rceil$ and let $x,x'$ (resp. $y,y'$) be the endpoints of 
the path $S_{1}$ (resp. $S_{m'}$) such that the one of the two $(x,y)$-paths (resp. $(x',y')$-paths) in $C_{r}$ contains both $x',y'$ ($x,y$) and the other, say $P$ (resp. $P'$),
contains none of them. Let $D_{S}$ be the closed-interior of the cycle $S_{1}\cup P'\cup S_{m'}\cup P$.
Let also $A_{C}$ be the closed annulus defined by the cycles $C_{\sigma^{\rm max}-(m'-1)}$
and $C_{\sigma^{\rm max}}$. Let $\Delta$ be any of the two connected components of $D_{S}\cap A_{C}$.
We now consider the graph
$$(L\cup\cupall {\cal C})\cap \Delta.$$
It is now easy to verify  that the above graph is the realization $G_{\cal U}$ 
of a tilted grid ${\cal U}=({\cal X},{\cal Z})$ of capacity $m'$, where 
the paths in ${\cal X}$ are 
the portions of the cycles $C_{\sigma^{\rm max}-(m'-1)},\ldots,C_{\sigma^{\rm max}}$
cropped by $\Delta$, while the paths in ${\cal Z}$ 
are the portions of the paths in $\{S_{1},\ldots,S_{m'}\}$
cropped by $\Delta$ (see Figure~\ref{fsig:thiss2sones}).
As ${\cal S}$ is an equivalence class of $\parallel$, it follows that ${\cal U}$ is  $L$-tidy, as required. 
\end{proof}

\subsection{Replacing linkages by cheaper ones}

\label{llt38f}

In this section we prove that a linkage $L$ of $k$ paths 
can be rerouted 
to a cheaper one, given  
the existence of an $L$-tidy tilted grid 
of capacity greater than $2^{k}$.
Given that $L$ is a cheap linkage, this will 
imply an exponential upper bound on the capacity of 
an  $L$-tidy tilted grid.

\medskip

Let $G$ be a plane graph and let $L$ be a linkage in $G$.
Let also $D$ be a closed disk in the surface where 
$G$ is embedded. We say that $L$ {\em crosses vertically}
$D$ if the outerplanar graph defined by the boundary 
of $D$ and $L\cap D$ has exactly two simplicial faces.
This naturally partitions the vertices of $\bnd(D)\cap L$ 
into the {\em up} and {\em down} ones. 
The following proposition is implicit 
in the proof of 
Theorem~2 in~\cite{AdlerKT16plan} (see the derivation 
of the unique claim in the proof of the former theorem). 
See also~\cite{EricksonN11shor} for related results.

\begin{proposition}
\label{taken_from_IPEC_11_paper}
Let $G$ be a plane graph and let $D$ be a closed disk and a linkage $L$
of $G$ of order $k$ that crosses $D$ vertically. Let also $L\cap D$ consist
of $r> 2^k$ lines.
Then there is a collection ${\cal N}$ of strictly less than $r$ mutually non-crossing lines in $D$
each connecting two points of $\bnd(D)\cap L,$ 
such that there exists some linkage  $R$ that is a subgraph of  $L\setminus \int(D)$
such that  $R\cup \cupall {\cal N}$ is a  linkage 
of the graph $(G\setminus D)\cup\cupall {\cal N}$ that is equivalent to $L$. 
\end{proposition}


\begin{figure}[ht]
\begin{center}
\includegraphics[width=8.3cm]{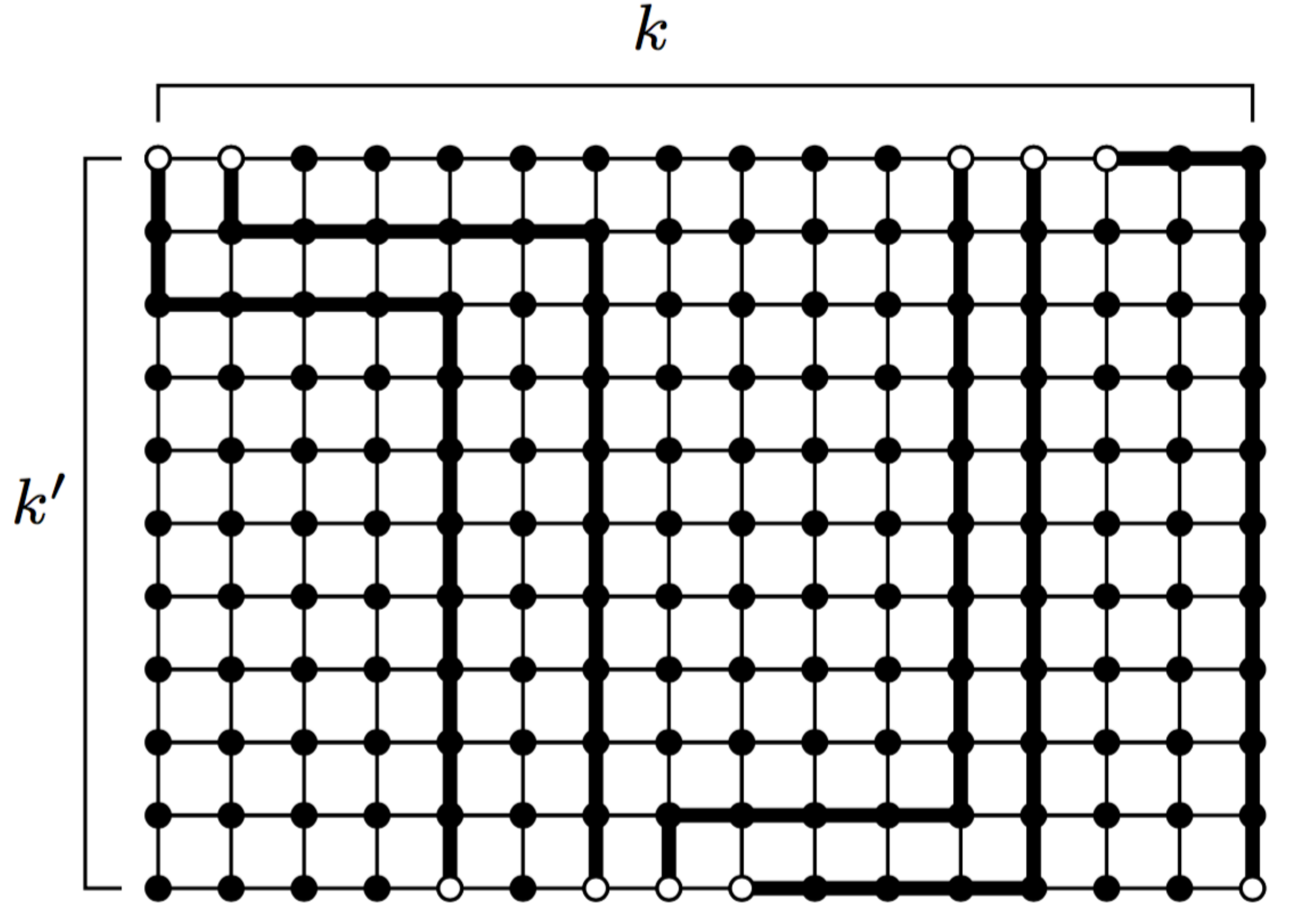}
\end{center}
\caption{An example of the proof of Lemma~\ref{upside}, where $k=16$, $ k'=11$, and $\rho=5$. The white vertices of the higher (resp. lower) horizontal line are the vertices in  $\{p_{1}^{\rm up},\ldots,p_{5}^{\rm up}\}$
(resp. $\{p_{1}^{\rm down},\ldots,p_{5}^{\rm down}\})$.}
\label{fig:h6jkdsld4r}
\end{figure}

\begin{lemma}
\label{upside}
 Let $k,k',\rho$ be integers such that  $0\leq \rho\leq k'\leq k$.
Let $\Gamma$ be a $(k\times k')$-grid
and let $\{p_{1}^{\rm up},\ldots,p_{\rho}^{\rm up}\}$
(resp. $\{p_{1}^{\rm down},\ldots,p_{\rho}^{\rm down}\}$)  be vertices of the higher  (resp. lower) horizontal line arranged as they appear in it from left to right.   
Then the grid $\Gamma$ contains $\rho$ pairwise disjoint paths $P_{1},\ldots,P_{\rho}$
such that, for every $h\in[\rho]$, the endpoints of $P_{h}$ are $p_{h}^{\rm up}$ 
and $p_{h}^{\rm down}$.
  \end{lemma}

\begin{proof}
We use induction on $\rho$. Clearly the lemma is obvious when $\rho=0$.
Let $(i,j)\in[k]^{2}$  such that $p_{\rho}^{\rm up}$  (resp. $p_{\rho}^{\rm down})$
is the $i$-th  (resp. $j$-th) vertex of the higher (\rm lower) horizontal line counting from left to right.
We examine first the case where 
$i\geq j$. Let $P_{\rho}$ be the path
created by starting from $p_{\rho}^{\rm up},$ 
moving $k'-1$ edges down, and then $i-j$
edges to the left.
For $h\in[\rho-1]$ let $P^{(\rm down)\prime}_{h}$
be the path created by starting from $p^{\rm down}_{h}$ and moving one edge up (clearly, $P^{\rm (down)\prime}_{h}$ consists of a single edge). We also denote by $p^{\rm (down)\prime}_{i}$ the other endpoint of $P^{\rm (down)\prime}_{i}$.
We now define $\Gamma'$ as the subgrid 
of $\Gamma$ that  occurs from $\Gamma$ after removing its lower horizontal line and,  for every $h\in[i,k]$, its $h$-th vertical line. By construction, none of the edges or vertices of $P_{\rho}$ belongs in $\Gamma'$.
Notice also that the higher (resp. lower) horizontal 
line of $\Gamma'$ contains 
all vertices in $\{p_{1}^{\rm up},\ldots,p_{\rho-1}^{\rm up}\}$
(resp.  $\{p_{1}^{\rm (down)\prime},\ldots,p_{\rho-1}^{\rm (down)\prime}\}$).
From the induction hypothesis,  $\Gamma'$ contains $\rho-1$ pairwise disjoint paths $P_{1}',\ldots,P_{\rho-1}'$
such that for every $h\in[\rho-1]$, the endpoints of $P_{h}$ are $p_{h}^{\rm up}$ 
and $p_{h}^{\rm (down)\prime}$.
It is now easy to verify that 
$P_{1}'\cup P_{1}^{\rm (down)\prime},\ldots,P_{\rho-1}'\cup P_{\rho-1}^{\rm (down)\prime}, P_{\rho}$ is the required collection of pairwise disjoint paths. 
For the case where $i<j$, just reverse the same grid upside down and the proof is identical (see Figure~\ref{fig:h6jkdsld4r}).
\end{proof}

\begin{lemma}
\label{h6jkld4r}
Let $\Gamma$ be a $(k\times k)$-grid embedded in the plane
and assume  
that the vertices of its outer cycle,  arranged in clockwise order, are:
$$\{v_1^{\rm up},\ldots,v_k^{\rm up},v_{2}^{\rm right},
\ldots,v_{k-1}^{\rm right},v_{k}^{\rm down},\ldots,v_{1}^{\rm down},
v_{k-1}^{\rm left},\ldots,v_{2}^{\rm left},v_{1}^{\rm up}\}.$$
Let also $H$ be a graph 
whose vertices have degree $0$ or $1$ and they 
can be cyclically arranged in clockwise order as $$\{x_1^{\rm up},\ldots,x_k^{\rm up},
x_k^{\rm down},\ldots,x_{1}^{\rm down},x_{1}^{\rm up}\}$$ such that if we add to $H$ the edges 
formed by pairs of consecutive vertices in this cyclic ordering, the resulting graph $H^{+}$ is outerplanar. Let $V^1$ be the vertices of $H$ that have degree 1 and let 
 $H^1=H[V^{1}]$. Then $H^1$ is a topological minor of $\Gamma$ via some pair
$(\phi_{0},\phi_{1}),$ satisfying the following properties:
\begin{enumerate}
\item $\phi_0(x_{i}^{\rm up})=v_{i}^{\rm up}$, $i\in\{1,\ldots,k\}\cap V^{1}$
\item $\phi_0(x_{i}^{\rm down})=v_{i}^{\rm down}$, $i\in\{1,\ldots,k\}\cap V^{1}$.
\end{enumerate}
\end{lemma}

\begin{proof} Let $U=\{x_1^{\rm up},\ldots,x_k^{\rm up}\}\cap V^1$
and $D=\{x_1^{\rm down},\ldots,x_{k}^{\rm down}\}\cup V^1$.
We define $\phi_{0}$ as in the statement of the lemma.
In the rest of the proof we provide the definition of  $\phi_{1}$.
We partition the  edges of $H^1$ into three sets:
the {\em upper edges} $E_{U}$ that connect vertices in 
$U$, the {\em down edges} $E_{L}$
that connect vertices in  $D$,
and the {\em crosssing edges} $E_{C}$ that have one endpoint in $U$ and one 
in $D$. As $|V(H^1)|\leq 2k$  we obtain that $|E(H^1)|\leq k$ and therefore  $|E_U|+|E_D|+|E_{C}|=|E(H^1)|\leq k$.
We set  $\rho=|E_{C}|$. 

We recursively define the {\em depth} of an  
edge $\{x_{i}^{\rm up},x_{j}^{\rm up}\}$ in $E_{U}$  as follows: it is  $0$ if there is no edge of $E_{U}$ with an endpoint 
in $\{x_{i+1}^{\rm up},\ldots,x_{j-1}^{\rm up}\}$
and is $i>0$ if the maximum depth of an edge with an endpoint in $\{x_{i+1}^{\rm up},\ldots,x_{j-1}^{\rm up}\}$ is $i-1$. The {\em depth} of an edge 
$\{x_{i}^{\rm down},x_{j}^{\rm down}\}$
is defined analogously. It directly follows, by the definition  of depth that:
%
%
%
\begin{eqnarray}
q^{\rm up}=\max\{\mbox{\em depth}(e)\mid e\in E_{U}\}+1 & \leq & |E_{U}|\label{up5rt}\\
q^{\rm down}=\max\{\mbox{\em depth}(e)\mid e\in E_{D}\}+1 & \leq & |E_{D}|\label{ddf5rt}
\end{eqnarray}
\begin{figure}[ht]
\begin{center}
\includegraphics[width=15cm]{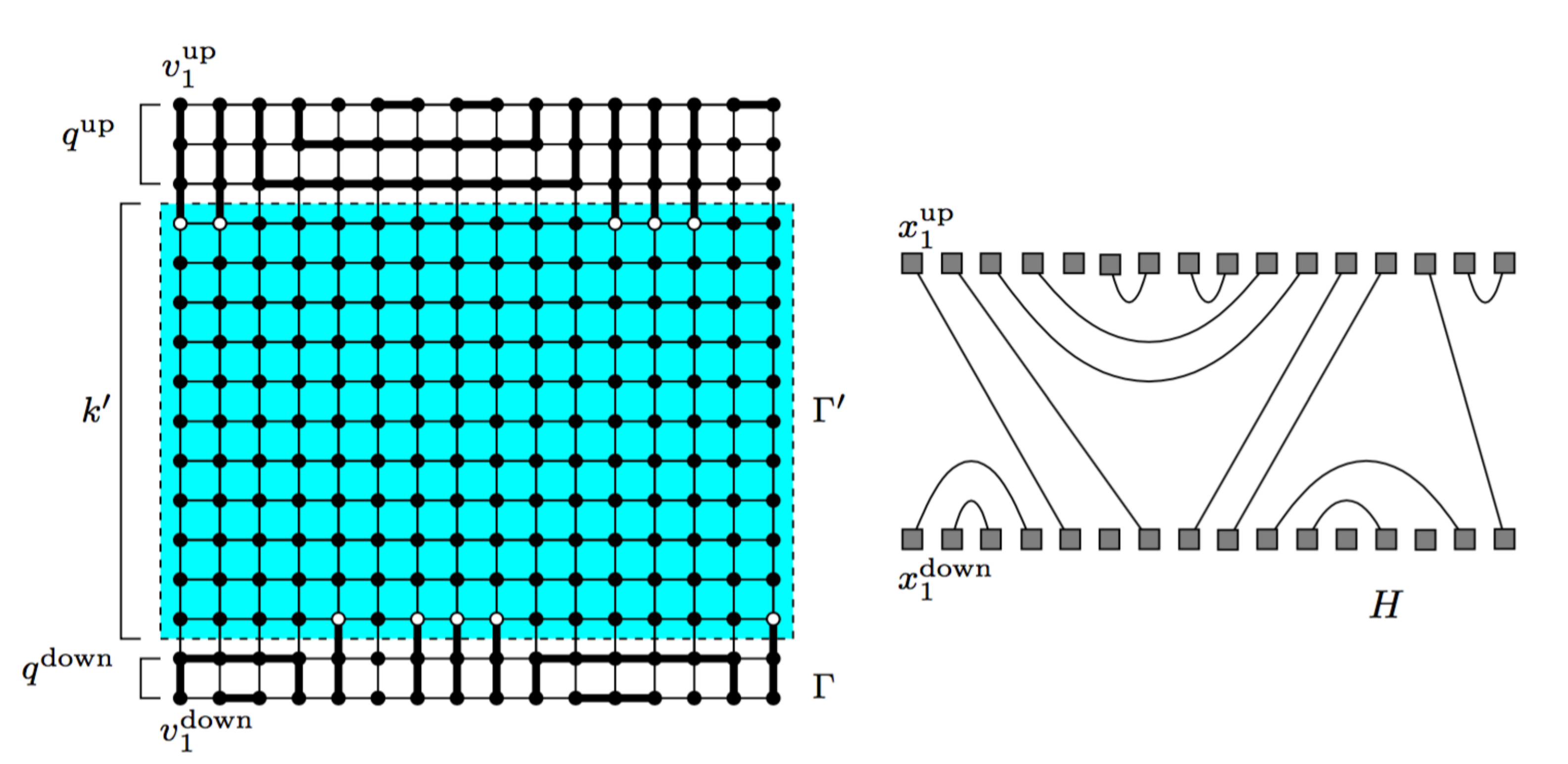}
\end{center}
\caption{An example of the proof of Lemma~\ref{h6jkld4r}.
On the left, the $(16\times 16)$-grid $G$ is depicted along with the way the graph $H$
(depicted on the right) is (partially) routed in it.
In the figure $q^{\rm up}=3$, $q^{\rm down}=2$,  $k'=11$, and $\rho=5$.}
\label{fig:h6jkld4r}
\end{figure}
We now continue with the definition of 
$\phi_{1}$ as follows:
\begin{itemize}
\item for every edge  $e=\{x_{i}^{\rm up},x_{j}^{\rm up}\}$ in $E_{U}$,
of depth $l$ and such that $i<j$, let $\phi_{1}(e)$ be the path defined if we start in the grid $G$ 
from $v_{i}^{\rm up},$ move $l$ steps down, then $j-i$ steps to the right, and finally move $l$ steps up to the vertex $v_{j}^{\rm up}$
(by ``number of steps'' we mean number of edges traversed).
\item for every edge  $e=\{x_{i}^{\rm down},x_{j}^{\rm down}\}$ in $E_{D}$,
of depth $l$ and such that $i<j$, let $\phi_{1}(e)$ be the path defined if we start in the grid $G$ 
from $v_{i}^{\rm down},$ move $l$ steps up, then $j-i$ steps to the right, and finally move $l$ steps down to the vertex $v_{j}^{\rm down}$.
\end{itemize}
Notice that the above two steps define the values of $\phi_{1}$ 
for all the upper and down edges. 
The construction guarantees that all paths in $\phi_1(E_{U}\cup E_{D})$ are mutually 
non-crossing.
Also, the distance between $\phi_0(U)$ and some  horizontal line 
of $\Gamma$ that contains edges of the images of the upper edges
is $\max\{\mbox{\em depth}(e)\mid e\in E_{U}\}$ that, from~\eqref{up5rt}, is equal  to $q^{\rm up}-1$.
Symmetrically, using~\eqref{ddf5rt} instead of~\eqref{up5rt}, the distance between $\phi_{0}(D)$ and the  horizontal lines 
of $\Gamma$ that contain edges of the images of the down edges
is equal to   $q^{\rm down}-1$. 
As a consequence,  the graph $$\Gamma'=\Gamma\setminus \{x\in V(\Gamma)\mid {\bf dist}_{\Gamma}(x,\phi_0(U))
<  q^{\rm up}\ \vee\  {\bf dist}_{\Gamma}(x,\phi_0(D))< q^{\rm down}\}$$
is a $(k\times k')$-grid $\Gamma'$, where $k'=k-(q^{\rm up}+q^{\rm down})$, whose 
vertices do not appear in any of the paths in  $\phi_{1}(E_{U}\cup E_{D})$. 
Given a crossing edge $e=\{x^{\rm up}_{i},x^{\rm down}_{j}\}\in E_{C}$, we define
the path $P_{e}^{\rm up}$ as the subpath
of $\Gamma$ created if we start from  $x^{\rm up}_{i}$ and then go $q^{\rm up}$ steps down. Similarly, we define $P_{e}^{\rm down}$ as the  subpath
of $\Gamma$ created if we start from  $x^{\rm down}_{j}$ and then go $q^{\rm down}$ steps up.  Notice that each of the paths 
 $P_{e}^{\rm up}$ (resp. $P_{e}^{\rm down}$) 
 share only one vertex, say $p_{e}^{\rm up}$ (resp. $p_{e}^{\rm down}$), with $\Gamma'$ 
that  is one of their endpoints (these endpoints are depicted as white vertices in the example of Figure~\ref{fig:h6jkld4r}). 
We use the notation $\{p_{1}^{\rm up},\ldots,p_{\rho}^{\rm up}\}$ (resp.  $\{p_{1}^{\rm down},\ldots,p_{\rho}^{\rm down}\}$)  for the vertices of the set $\{p_{e}^{\rm up}\mid e\in E_{C}\}$ (resp.  $\{p_{e}^{\rm down}\mid e\in E_{C}\}$) such that, 
for every $h\in[\rho]$, there exists 
an $e\in E_{C}$ such that 
$p^{\rm up}_{h}$ is an endpoint of $P_{e}^{\rm up}$ and $p^{\rm down}_{h}$ is an endpoint of $P_{e}^{\rm down}$.
We also agree that the vertices in $\{p_{1}^{\rm up},\ldots,p_{\rho}^{\rm up}\}$ (resp.  $\{p_{1}^{\rm down},\ldots,p_{\rho}^{\rm down}\}$) are ordered as they appear from left to right in the upper (lower) horizontal line of $\Gamma'$ (this is possible because of the outeplanarity of $H^{+}$).

Notice that $\rho=|E(H^{1})|-(|E_{U}|+|E_{D}|)\leq k-(|E_{U}|+|E_{D}|)$ which by~\eqref{up5rt} and~\eqref{ddf5rt} implies that  $\rho\leq k'$.

As $\rho\leq k'\leq k$, we can now apply 
Lemma~\ref{upside} on $\Gamma'$,
$\{p_{1}^{\rm up},\ldots,p_{\rho}^{\rm up}\}$ and $\{p_{1}^{\rm down},\ldots,p_{\rho}^{\rm down}\}$ and obtain a collection $\{P_{e}\mid e\in E_{C}\}$ of $\rho$ pairwise disjoint paths in $\Gamma'$
between the 
vertices of $\{p_{e}^{\rm up}\mid e\in E_{C}\}$ and the vertices of $\{p_{e}^{\rm down}\mid e\in E_{C}\}$.
It is now easy to verify that 
$\{P_{e}^{\rm up}\cup P_{e}\cup P_{e}^{\rm down}\mid e\in E_{C}\}$ is a collection of 
$\rho$ vertex disjoint paths between $U$ and $D$. We can now complete the definition of $\phi_{1}$ for the crossing edges of $H$ by setting, for each $e\in E_{C}$,
$\phi(e)=P_{e}^{\rm up}\cup P_{e}\cup P_{e}^{\rm down}$. By the above construction it is clear that $(\phi_{1},\phi_{2})$ provides the claimed topological isomorphism.
%
\end{proof}

\begin{lemma}
\label{ofiisf8}
Let $G$ be a graph with a linkage $L$ consisting of $k$ paths.
Let also ${\cal U}=({\cal X},{\cal Z})$ be an $L$-tidy tilted grid 
of $G$ with capacity $m$. 
Let also $\Delta$ be the closed-interior of  the perimeter of $G_{\cal U}$.
If $m> 2^{k}$, then
$G$ contains a  linkage $L'$ 
such that 
\begin{enumerate}
\item $L$ and $L'$ are equivalent,
\item $L'\setminus \Delta\subseteq L\setminus \Delta$, and 
\item $|E(\cupall {\cal Z}\cap L')| <  |E(\cupall {\cal Z}\cap L)|$.
\end{enumerate}
\end{lemma}

\begin{proof}
We use the notation ${\cal X}=\{X_{1},\ldots,X_{m}\}$ and ${\cal Z}=\{Z_{1},\ldots,Z_{m}\}$.
Let $G_{\cal U}$ be the realization of ${\cal U}$ in $G$ and let $G^*$ (resp. $L^*$)
be the graph (resp. linkage) obtained from $G$  (resp. $L$) if we contract all edges 
in the paths of  $\bigcup_{(i,j)\in\{1,\ldots,r\}^{2}}I_{i,j}$, where 
$I_{i,j}=X_{i}\cap Z_{j}$, $i,j\in\{1,\ldots,m\}$. 
We also define ${\cal X}^*$ and ${\cal Z}^*$ by applying the same contractions to their paths.
Notice that  ${\cal U}^*=({\cal X}^*,{\cal Z}^*)$ is an $L^*$-tidy tilted grid 
of $G^*$ with capacity $m$ and that the lemma follows if
we find a linkage $L'^*$
such that  the above three conditions are true for $\Delta^{*},L^*,L'^*,$ and ${\cal Z}^{*}$, where $\Delta^{*}$ is the closed-interior of the perimeter of $G_{\cal U}^{*}$ (recall that $G_{\cal U}^{*}$ is the representation of ${\cal U}$ that is isomorphic to $G_{{\cal U}^{*}}$).

Let $G^{*-}=(G^{*}\setminus \Delta^*)\cup \cupall {\cal Z}$
and apply Proposition~\ref{taken_from_IPEC_11_paper}
on $G^{*-}$, $\Delta^{*}$, and $L^{*}$. Let  ${\cal N}$ 
be a collection of strictly less than $m$ mutually non-crossing lines in $D$
each connecting two points of $\bnd(\Delta^*)\cap L^*$ and a linkage 
$R\subseteq L^*\setminus \int(\Delta^*)$ 
such that $L_0=R\cup \cupall {\cal N}$ is a  linkage 
of the graph $(G^*\setminus \Delta^*)\cup\cupall {\cal N}$ that is equivalent to $L^*$. Let $H=(L_{0}\cap \Delta^{*})\cup (L^{*}\cap \bnd(\Delta^{*}))$. 
Notice that in $H$, the set 
$V(L_{0}\cap \Delta^{*})$ contains the vertices 
of $H$ of degree 1 while  the rest of the vertices of $H$ have  degree 0 and all edges of $H$ have their endpoints in $V(L_{0}\cap \Delta^{*})$.
Recall that  the $(m\times m)$-grid $\Gamma$ is a topological minor of $G_{\cal U}^*$ 
via some pair $(\chi_{0},\chi_{1})$ satisfying the conditions A and B in the definition of 
tilted grid.

We are now in position to apply Lemma~\ref{h6jkld4r} for the $(m\times m)$-grid $\Gamma$ and $H$.
We obtain  that $H^{1}=L_{0}\cap \Delta^{*}$ is a topological minor of $\Gamma$ via some pair
$(\phi_{0},\phi_{1})$. We now define the graph 
$$L=\bigcup_{e\in E(H^1)}E(\phi_{1}(e)).$$ Notice that $L$ is a subgraph of $\Gamma$. We also define  
the graph 
$$Q=\bigcup_{e\in E(L)}\chi_{1}(e)$$ which, in turn,  is a subgraph of $G_{\cal U}^*$.
Observe that 
 $L'^*=R\cup Q$
is a linkage of $G^*$ that is equivalent to $L^{*}$. This proves Condition~1.
Condition 2 follows from the fact that  $R\subseteq L^*\setminus \int(\Delta^*)$.
Notice now that, as $|{\cal N}|<m$, $E(\cupall{\cal Z}^*\cap Q)$
is a proper subset of $E(\cupall {\cal Z}^*)$.
By construction of $L'^*$, it holds that  $E(\cupall {\cal Z}\cap L'^*) =E(\cupall{\cal Z}\cap Q)$.
Moreover, as  ${\cal U}^*=({\cal X}^*,{\cal Z}^*)$ is an $L^*$-tidy tilted grid 
of $G^*$, it follows that $E(\cupall {\cal Z}^*)=E(\cupall {\cal Z}^*\cap L^*)$. Therefore, Condition~3 follows.
\end{proof}

%

\subsection{Existence of an irrelevant vertex}
\label{subsec:irre}
%
%

We now bring together all results from 
the previous subsections in order to prove Theorem~\ref{main}.

\begin{lemma} 
\label{gthwdsgosd}
There exists an algorithm that, given an instance 
$(G,{\cal P}=\{(s_{i},t_{i})\in V(G)^{2}, i\in\{1,\ldots,k\}\})$ of   \PDP,  
either outputs a tree-decomposition of $G$ of width at most $9\cdot (k\cdot 2^{k+2}+1)\cdot \lceil \sqrt{2k+1}\, \rceil$ or outputs 
an irrelevant  vertex $x\in V(G)$ for $(G,{\cal P})$. This algorithm runs in $2^{2^{O(k)}}\cdot n$ steps.
\end{lemma}

\begin{proof}
Let $T=\{s_{1},\ldots,s_{k},t_{1},\ldots,t_{k}\}$. 
By applying the  algorithm of Lemma~\ref{u82mnb1o}, for $r=k\cdot 2^{k+2}$ 
either 
we output a tree-decomposition of $G$ of width at most 
$9(r+1)\cdot \lceil \sqrt{2k+1}\, \rceil )$ or we 
find an internally chordless  cycle $C$ 
of $G$ such that  $G$ contains a tight sequence of 
cycles  ${\cal C}=\{C_{0},\ldots, C_{r}\}$ in $G$ where 
$C_{0}=C$ and 
all vertices of $T$
are in the open exterior of $C_{r}$. From  Lemma~\ref{u82mnb1o}, this can be done 
in $2^{(r\cdot\sqrt{|T|})^{O(1)}} \cdot n =  2^{2^{O(k)}}\cdot n$  steps.

Assume that $G$ has a linkage whose pattern is 
${\cal P}$ and, among all such linkages, let $L$ be a ${\cal C}$-cheap one.
Our aim is to prove that $V(L\cap C_{0})=\varnothing,$
i.e., we may pick $x$ to be any of the vertices in $D_{0}$. 

First, we can assume that $k\geq 2$. Otherwise,  if $k=1$,
the fact that $L$ is ${\cal C}$-cheap, implies  that
  $L\cap D_{r-1}=\varnothing\Rightarrow L\cap D_{0}=\varnothing$ and we are done.

For every $i\in\{0,\ldots,r\}$, we 
define ${\cal Q}^{(i)}=({\cal C}^{(i)},L^{(i)})$
where ${\cal C}^{(i)}=\{C_{0},\ldots,C_{i}\}$
and $L^{(i)}$ is the subgraph of $L$ consisting of the union of 
the connected components 
of $L$ that  have common points with $D_{i}$.
As $r+1>k$, at least one of ${\cal Q}^{(i)}, i\in\{0,\ldots,r\}$
is touch-free. Let ${\cal Q}'=({\cal C}',L')$ be the 
touch-free CL-configuration in $\{{\cal Q}^{(1)},\ldots,{\cal Q}^{(r)}\}$ of the 
highest index, say  $h$. In other words, ${\cal C}'={\cal C}^{(h)}$ and $L'=L^{(h)}$.
Moreover, ${\cal C}'$ is tight in $G$
and $L'$ is ${\cal C}'$-cheap. Let $k'$ be the number of connected components 
of $L'$. 
We set $d=r-h$ and observe that $k'\leq k-d$,  
while ${\cal C}'$ has $r'=r+1-d>0$ concentric cycles.
Again, we  assume that $k'\geq 2$ as, otherwise,  the fact that $L'$ is ${\cal C}'$-cheap
implies that
  $L'\cap D_{r'-1}=\varnothing\Rightarrow L'\cap D_{0}=\varnothing$ and we are done. Therefore $0\leq d\leq k-2$.

%
%
As  ${\cal C}'$ is tight in $G$
and $L'$ is ${\cal C}'$-cheap, by Lemma~\ref{DiscTheorem1},
${\cal Q}'$ is convex.
To prove  that $V(L\cap C_{0})=\varnothing$ it is enough to show 
that all segments of ${\cal Q}$ have positive eccentricity
and for this it is sufficient to prove  that  all segments of ${\cal Q}'$ have positive eccentricity. 
Assume to the contrary that some segment $P_0$ of ${\cal Q}'$
has eccentricity $0$. Then, from the third condition in the definition of convexity
we can derive the existence of a sequence $P_{0},\ldots,P_{r'-1}$
of segments such that for each $i\in\{0,\ldots,r'-1\}$, $P_{i+1}$ is 
inside the zone of $P_{i}$. This implies the existence in the segment tree $T({\cal Q}')$
of a path of length  $r'$ from its root to one of its leaves, therefore 
$T({\cal Q}')$ has height $r'$. By  Lemma~\ref{theo:bounding_number_of_types}, the 
real height of $T({\cal Q}')$ is at most $2k'-3$. By Observation~\ref{ogll5fk}, the dilation of $T({\cal Q}')$
is at least $\frac{r'}{2k'-3}\geq \frac{k\cdot 2^{k+2}-d}{2k-2d}> \frac{k\cdot 2^{k+2}}{2k}=2^{k+1}$. 
By Observation~\ref{ogsdll5dfk} and Lemma~\ref{theo:bounding_more},
$G$ contains an $L'$-tidy tilted grid ${\cal U}=({\cal X},{\cal Z})$ of capacity $>2^{k}$.
From~Lemma~\ref{ofiisf8}, $G$ contains another linkage $L''$ with the same pattern as $L'$
and such that $c(L'')<c(L')$, a contradiction to the fact that $L'$ is ${\cal C}'$-cheap.

Since $V(L\cap C_{0})=\varnothing,$ any vertex of $G$ in the 
closed-interior of $C_0$ is irrelevant. 
%
\end{proof}

\noindent{{\em Proof of Theorem~\ref{main}}.
The proof follows from Lemma~\ref{gthwdsgosd}, 
\href{http://files.thilikos.info/data/various/result_check.png}{taking into account} that, for every $k\geq 1$, 
$$\href{http://files.thilikos.info/data/various/check.py}{82\cdot k^{3/2}\cdot 2^{k}>9\cdot (k\cdot 2^{k+2}+1)\cdot \lceil \sqrt{2k+1}\, \rceil )}$$

\vspace{-3mm}
\qed

%
%
%
%

\section{An algorithm for \PDP}
\label{sec:algo}

In this section we prove Theorem~\ref{theorem:algo}. 
In particular, we briefly describe an algorithm that,
given an instance $(G,{\cal P})$ of \DP\ where $G$ is planar, provides 
a solution to \PDP, if one exists,   in $2^{2^{O(k)}}\cdot n^{O(1)}$ steps.

Our algorithm is based on the following proposition.

%
%

\begin{proposition}[\!\!\cite{Scheffler94apra}]
\label{tpoolll}
There exists an algorithm that, given an instance $(G,{\cal P})$ of \PDP\  
and a tree decomposition of $G$ of width at most $w$, 
either  reports that $(G,{\cal P})$ is a NO-instance or  outputs 
a solution of \PDP\ for $(G,{\cal P})$ in $2^{O(w\log w)}\cdot n$ steps.
\end{proposition}

\begin{proof}[Proof of Theorem~\ref{theorem:algo}]
By applying the algorithm of Lemma~\ref{gthwdsgosd},  
we either  find an irrelevant vertex $v$  for $(G,{\cal P})$ or we 
obtain a tree-decomposition of $G$ of width $2^{O(k)}$.
In the first case,  we again look for an irrelevant vertex in the equivalent 
instance   $(G,{\cal P})\leftarrow (G\setminus v,{\cal P})$.
This loop
breaks if the second case appears, namely when 
 a tree decomposition of  $G$ of 
width $2^{O(k)}$ is found.
Then we apply the algorithm of Proposition~\ref{tpoolll},
that solves the problem in $2^{2^{O(k)}}\cdot n$ steps.
As, unavoidably, the loop will break in less than $n$ steps, the claimed running time follows.
\end{proof}

\noindent
{\bf Acknowledgment.} We thank Ken-ichi Kawarabayashi and Paul Wollan for
providing details on the bounds in~\cite{KawarabayashiW2010asho}.  We are particularly thankful to the two anonymous referees for their detailed and insightful reviews
that helped us to considerably improve the paper.

%
\end{document}